\documentclass[leqno,11pt]{amsart} 
\usepackage{amssymb} 
\usepackage{mathrsfs} 
\usepackage[pdftex]{graphicx}
\usepackage{tikz} 
\usetikzlibrary{patterns, intersections, calc}  
\theoremstyle{plain} 
\newtheorem{theorem}{\indent\sc Theorem}[section]
\newtheorem{lemma}[theorem]{\indent\sc Lemma}
\newtheorem{corollary}[theorem]{\indent\sc Corollary}
\newtheorem{proposition}[theorem]{\indent\sc Proposition}

\newtheorem{fact}[theorem]{\indent\sc Fact}  
\theoremstyle{definition} 
\newtheorem{definition}[theorem]{\indent\sc Definition}
\newtheorem{remark}[theorem]{\indent\sc Remark}

\def\C{{\mathbb{C}}}
\def\RC{{\widehat{{\mathbb{C}}}}}
\def\D{{\mathbb{D}}}
\def\F{{\mathscr{F}}}
\def\G{{\mathscr{G}}}
\def\H{{\mathbb{H}}}
\def\L{{\mathbb{L}}}
\def\Lc{{\mathcal{L}}}
\def\M{{\mathscr{M}}}
\def\P{{\mathscr{P}}}
\def\Pi{{\mathbb{P}}}
\def\R{{\mathbb{R}}}
\def\Z{{\mathbb{Z}}}

\let\Re\relax
\DeclareMathOperator{\Re}{Re}
\let\Im\relax
\DeclareMathOperator{\Im}{Im} 

\begin{document}

\title[Bloch--Ros principle]{Bloch--Ros principle \\ and its application to surface theory} 

\author[S.~Kasao]{Shunsuke Kasao} 

\author[Y.~Kawakami]{Yu Kawakami} 

\dedicatory{Dedicated to Professor Toshihiro Nakanishi on his sixty-fifth birthday}


\renewcommand{\thefootnote}{\fnsymbol{footnote}}
\footnote[0]{2020\textit{ Mathematics Subject Classification}.
Primary 53A10; Secondary 30D35, 30D45.}
\keywords{
Gauss map, normal family, curvature estimate, Bloch--Ros principle}
\thanks{
This work was supported by JSPS KAKENHI Grant Number JP23K03086. 
}

\address{
Shibaura institute of technology Kashiwa, \endgraf
junior and senior high school, \endgraf
Masuo, Kashiwa, 277-0033, \endgraf
Japan
}
\email{kasao15039@gmail.com}

\address{
Faculty of Mathematics and Physics, \endgraf
Kanazawa University \endgraf
Kanazawa, 920-1192, \endgraf
Japan
}
\email{y-kwkami@se.kanazawa-u.ac.jp}


\maketitle

\begin{abstract}
There exists the duality between normal family theory and value distribution theory of meromorphic functions, which is called the Bloch principle. Zalcman formulated a more precise statement on it. In this paper, based on the Zalcman and Ros works, we comprehend the phenomenon of the trinity among normal family theory, value distribution theory and minimal surface theory and give a systematic description to the relationship among the Montel theorem, the Liuoville theorem and the Bernstein theorem as well as the Carath\'{e}odory--Montel theorem, the Picard little theorem and the Fujimoto theorem. We call this phenomenon Bloch--Ros principle. We also generalize the Bloch--Ros principle to various classes of surfaces, for instance, maxfaces in the Lorentz--Minkowski $3$-space, improper affine fronts in the affine $3$-space and flat fronts in the hyperbolic $3$-space. In particular, we give an effective criterion to determine which properties for meromorphic functions that play a role of the Gauss maps of these classes of surfaces satisfy the Gaussian curvature estimate.
\end{abstract}

\tableofcontents

\section{Introduction} 
The Bloch principle is a widely recognized guiding principle in complex analysis. Bloch \cite[p. 84]{Bl1926} pointed out that a family of meromorphic functions which possess a common property $P$ on a domain $\Omega$ 
in the complex plane $\C$ is likely to be normal on $\Omega$ if there does not exist nonconstant meromorphic function with this property $P$ on $\C$. A typical example of the Bloch principle 
is the correspondence between the Picard little theorem and the Carath\'eodory--Montel theorem. Based on this principle, Bloch was able to predict several important results such as 
the Ahlfors five islands theorem and the Cartan theorem on holomorphic curves omitting hyperplanes in Nevanlinna theory
(See \cite{Ko2003, Ne1970, NW2013, Ru2021} for Nevanlinna theory).
However, there exist some counterexamples to this principle (see \cite{RU1986}, \cite[Section 4.2]{Sc} for example). In the meantime, Zalcman \cite{Za1975} obtained an effective lemma to determine the normality for families of meromorphic functions and it indicates when the Bloch principle is valid. Notable references for the Bloch principle are \cite{WB, Sc, LZ}. 

Based on this formulation, Ros gave a similar formulation to minimal surface theory. Before we explain it, we recall the Bernstein theorem and the Fujimoto theorem.   
The Bernstein theorem states that any entire minimal graph in the Euclidean 
$3$-space $\R^3$ must be a plane. As one of the proofs of the Bernstein theorem, Heinz \cite{He1952} proved the Gaussian curvature estimate for minimal graphs in $\R^3$. 
Fujimoto \cite{FH1988} obtained the Gaussian curvature estimate for minimal surfaces in $\R^3$ whose Gauss maps omit $5$ or more points in the unit sphere $S^2$.  
As an application of this estimate, Fujimoto proved that any complete minimal surface in $\R^3$ whose Gauss map omits $5$ or more points in $S^2$ must be a plane. 
These results imply that the Gaussian curvature estimate occurs if the Gauss maps of  minimal surfaces satisfy a certain property $P$. By applying this estimate, 
we can prove that any complete minimal surface in $\R^3$ whose Gauss map has the property $P$ must be a plane. In \cite{RA}, Ros obtained an effective criterion to determine which properties for the Gauss maps of minimal surfaces in $\R^3$ satisfy the Gaussian curvature estimate. From this criterion, we can apply some results in normal family theory to not only value distribution theory of meromorphic functions on $\C$ but also minimal surface theory. In fact, Ros proved the Fujimoto results by using this criterion and it contributed to the similar result in the Euclidean space of higher dimension \cite{LP, OR}. We call this phenomenon the {\it Bloch--Ros principle} in honor of the Bloch principle and the Ros work. We thereby obtain the following trinity (Figure 1).  

\vspace{5mm}
\begin{center}
\begin{tikzpicture}[scale=0.7]
\draw(-10.5,0) rectangle (-2.5,1);
\draw(2.5,0) rectangle (10.5,1);
\draw(-4,-4) rectangle (4,-3);
\node (A) at (-6.5,0.5) {Normal Family};
\node (A2) at (-6.5,-0.5) {Carath\'{e}odory--Montel};
\node (B) at (6.5,0.5) {Meromorphic Function on $\mathbb{C}$};
\node (B2) at (6.5,-0.5) {Picard};
\node (C) at (0,-3.5) {Minimal Surface in $\mathbb{R}^3$};
\node (C2) at (0,-4.5) {Fujimoto};
\draw (-2.4,0.5)--(2.4,0.5)--(0,-2.9)--(-2.4,0.5);
\node (G) at (0,-0.3) {Bloch--Ros};
\node (H) at (0,-1) {Principle};
\node (I) at (0,-5.5) {FIGURE $1$. Bloch--Ros Principle};
\end{tikzpicture}
\end{center}

Value distribution theoretic properties of the Gauss maps are also valid for other classes of surfaces which may admit singularities. For instance, the second author and Nakajo \cite{KN} showed that any weakly complete improper affine front in the affine $3$-space $\R^3$ whose Lagrangian Gauss map omits $4$ or more points in the extended complex plane $\RC := \C \cup \{ \infty \}$ must be an elliptic paraboloid. 
As an application of this result, the parametric affine Bernstein theorem for improper affine spheres in $\R^3$ was proved. Moreover, the second author \cite{Ka2014}  
showed similar results for flat fronts in the hyperbolic $3$-space $\H^3$. In \cite{YK2013}, the second author gave a curvature estimate for the conformal metric 
$ds^2 =(1+|g|^2)^{m} \, |f|^2 \, |dz|^2$ on an open connected Riemann surface $\Sigma$, where $f \, dz$ is a holomorphic $1$-form on $\Sigma$, $g$ is a meromorphic function on $\Sigma$
that omits $m+3$ or more distinct values in $\RC$ and $m$ is a 
positive integer (\cite[Theorem 2.1]{YK2013}). In this paper, we call this triple $(\Sigma, \, f \, dz, \, g)$ 
a {\it Weierstrass $m$-triple} (Definition \ref{def:m-pair} in this paper). 
Moreover, by applying this estimate, the second author proved that the precise maximal number of omitted values of a nonconstant meromorphic function $g$ on $\Sigma$ with the complete conformal metric $ds^2$ is $m+2$ (\cite[Corollary 2.2, Proposition 2.4]{YK2013}). We thereby find that the maximal number of omitted values of the Gauss map of any nonflat complete minimal surface in $\R^3$ is $4 \, (=2+2)$ because $m=2$. Furthermore, the maximal number of omitted values of the Lagrangian Gauss map of any weakly complete improper affine front in $\R^3$ that is not an elliptic paraboloid is $3 \, (=1+2)$ because $m=1$. We also remark that the second author \cite{Ka2015} obtained other function-theoretic properties 
such as unicity theorem and the special case of the Ahlfors islands theorem for Weierstrass $m$-triples.      

The purpose of this paper is to establish a generalization of the Bloch--Ros principle to several classes of surfaces and to give a unified description of the phenomenon 
in normal family theory, value distribution theory and surface theory. The paper is organized as follows: In Section \ref{sec:2}, 
we refine the formulation of the Bloch principle in \cite{RA} from the viewpoint of complex analysis, for example, Proposition \ref{prop:meromorphically extension} and the Zalcman lemma 
(Theorem \ref{thm:Zalcman}). We first review the basic facts of normal family theory in Section \ref{sbsec:2.1}. We next define the closedness and the compactness of property in Definition \ref{def:cpt prop}. 
We state the difference between the normal property used in complex analysis (see \cite{WB,Sc,LZ} for example) and the compact property in Remark \ref{rem:cpt normal}. 
Proposition \ref{lem:cpt and const} and the Zalcman lemma (Theorem \ref{thm:Zalcman}) represent the duality between normal family theory and value distribution theory of meromorphic functions. 
As examples of compact property, we consider the property $P_L$ of being bounded in Proposition \ref{prop:Liouville Montel} and the property $P_{X}$ of omitting $3$ points in Theorem \ref{thm:Picard Montel}. 
In particular, we reprove the compactness of $P_{X}$ if the cardinality $|X|$ of a set $X (\subset \RC)$ is at least $3$ in \cite[Theorem 2]{RA}. We thereby see that the duality between a weaker result of the Montel theorem and the Liouville theorem and 
the duality between the Carath\'eodry--Montel theorem and the Picard little theorem.
Furthermore, we also obtain the contraposition of the Picard great theorem in Remark \ref{rem:Picard great thm}.
In Section \ref{sec:3}, we construct a generalization of the Bloch--Ros principle to Weierstrass $m$-triples. We first define a Weierstrass $m$-triple $(\Sigma, \, f \, dz, \, g)$ in Definition \ref{def:m-pair} and $m$-curvature estimate in Definition \ref{def:m-curvature est}. We next give an effective criterion to determine which compact properties satisfy $m$-curvature estimate in Theorem \ref{thm:m-curvature est}. The proof is given in Section \ref{sbseq:3.2}. As applications of this criterion, we give 
two examples (Theorems \ref{thm:Bernstein} and \ref{thm:Fujimoto}) of compact property that satisfies 
$m$-curvature estimate. By virtue of these estimates, we obtain the Liouville-type theorem (Corollary \ref{cor:Bernstein}) and the Picard-type theorem (Corollary \ref{cor:Fujimoto}) for the meromorphic function $g$ of a Weierstrass $m$-triple $(\Sigma, \, f \, dz, \, g)$ if its metric $ds^2 =(1+|g|^2)^{m} \, |f|^2 \, |dz|^2$ is complete on $\Sigma$. In Section \ref{sec:4}, by the generalization of the Bloch--Ros principle, we deduce value distribution theoretic properties such as uniqueness and the maximal number of omitted values for the Gauss maps of several classes of complete surfaces, that is, minimal surfaces in $\R^3$ (Section \ref{sbseq:4.1}), maxfaces in the Lorentz--Minkowski 3-space $\L^3$ (Section \ref{sbseq:4.2}), improper affine fronts in $\R^3$ (Section \ref{sbseq:4.3}) and flat fronts in $\H^3$ (Section \ref{sbseq:4.4}).  
We therefore find that the compact property $P_L$ corresponds to the uniqueness theorems such as the Bernstein theorem for minimal surfaces in $\R^3$ and the parametric affine Bernstein theorem for improper affine spheres in $\R^3$ and the compact property $P_{X} \, (|X| \geq m + 3)$ also implies the Fujimoto theorem for complete minimal surfaces in $\R^3$ and the result of \cite[Theorem 3.2]{KN} for weakly complete improper 
affine fronts in $\R^3$. In conclusion, we can obtain the following trinity (Figure 2). 
\vspace{5mm}
\begin{center}
\begin{tikzpicture}[scale=0.7]
\draw(-10.5,0) rectangle (-2.5,1);
\draw(2.5,0) rectangle (10.5,1);
\draw(-4,-6.1) rectangle (4,-3);
\node (A) at (-6.5,0.5) {Normal Family};
\node (A1) at (-6.5,-0.5) {Montel};
\node (A2) at (-6.5,-1.3) {Carath\'{e}odory--Montel};
\node (B) at (6.5,0.5) {Meromorphic Function on $\mathbb{C}$};
\node (B1) at (6.5,-0.5) {Liouville};
\node (B2) at (6.5,-1.3) {Picard};
\node (C1) at (0,-3.5) {Minimal Surface in $\mathbb{R}^3$};
\node (C2) at (0,-4.2) {Maxface in $\mathbb{L}^3$};
\node (C3) at (0,-4.9) {Improper Affine Front in $\mathbb{R}^3$};
\node (C4) at (0,-5.6) {Flat Front in $\mathbb{H}^3$};
\node (C5) at (0,-6.6) {Uniqueness (Bernstein, etc.)};
\node (C6) at (0,-7.4) {Maximal number of omitted values};
\node (C7) at (0,-8.1) {(Fujimoto, etc.)};
\draw (-2.4,0.5)--(2.4,0.5)--(0,-2.9)--(-2.4,0.5);
\draw[->] (-1,1)--(1,1);
\node (D) at (0,1.5) {\footnotesize{Proposition \ref{lem:cpt and const}}};
\draw[->] (1,0)--(-1,0);
\node (E) at (0,-0.5) {\footnotesize{Theorem \ref{thm:Zalcman}}};
\draw[->] (-2.1,-1)--(-1.1,-2.4);
\node (G) at (-3,-2.1) {\footnotesize{Theorem \ref{thm:m-curvature est}}};
\draw[->] (2.1,-1)--(1.1,-2.4);
\node (H) at (2.2,-2.1) {\footnotesize{\cite{YK2013}}};
\node (I) at (0,-9.1) {FIGURE $2$. Generalization of Bloch--Ros Principle};
\end{tikzpicture}
\end{center}


\section{Duality between normal family theory and value distribution theory} \label{sec:2} 

In this section, we refine the formulation of the Bloch principle in \cite{RA} from the viewpoint of complex analysis. 

\subsection{Normal family} \label{sbsec:2.1} \par
We summarize the notion of normal families of meromorphic functions on a domain in $\C$. See \cite{CTC,Sc} for details. 

We first define a half of the chordal distance $\chi(\cdot,\cdot)$ based on \cite[Section 1.3]{Fu1993}. The extended complex plane $\RC:=\C \cup \{ \infty \}$ is identified with the unit sphere $S^2(\subset \R^3)$ by the stereographic projection $\pi_N:\RC \longrightarrow S^2$ that maps $\infty \in \RC$ to $N=(0,0,1) \in S^2$.
Here, a distance on $S^2$ is the induced metric as a subset of $\R^3$. Then the distance $\chi(z_1,z_2)$ between $z_1$ and $z_2$ in $\RC$ is defined as a half of the distance on $S^2$. We call $\chi(z_1,z_2)$ a half of the \textit{chordal distance} between $z_1$ and $z_2$ or the \textit{spherical distance} between $z_1$ and $z_2$. For $z_1, \, z_2 \in \C$, we can express as follows:
\[
\chi(z_1, z_2) = \dfrac{|z_1-z_2|}{\sqrt{1+|z_1|^2} \,  \sqrt{1+|z_2|^2}}, \quad \chi(z_1, \infty) = \dfrac{1}{\sqrt{1+|z_1|^2}}. 
\]
By its definition, $(\RC, \chi)$ is a compact metric space. 

\begin{definition} \label{def:conv}
A sequence $\{ f_n \}_{n=1}^\infty$ of functions converges \textit{spherically uniformly} to a function $f$ on a set $E \subset \C$ or $\{ f_n \}_{n=1}^\infty$ converges uniformly to $f$ on $E$ with respect to the spherical distance
if, for any $\varepsilon >0$, there exists a positive integer $N \in \Z_{>0}$ such that $\chi(f_n(z),f(z)) < \varepsilon$ for all $n \geq N$ and for all $z \in E$.
Then we denote it by $f_n \stackrel{\mathrm{sph.}}{\rightrightarrows} f \; \mathrm{on}\; E$. 
Moreover, a sequence $\{ f_n \}_{n=1}^\infty$ of functions converges \textit{spherically uniformly} on compact subsets of a domain $\Omega \subset \C$ to a function $f$ or $\{ f_n \}_{n=1}^\infty$ converges locally uniformly to $f$ on $\Omega$ with respect to the spherical distance
if, for any compact subset $K \subset \Omega$ and for any $\varepsilon >0$, there exists $N \in \Z_{>0}$ such that $\chi(f_n(z),f(z)) < \varepsilon$ for all $n \geq N$ and for all $z \in K$.
We also define it by $f_n \stackrel{\mathrm{sph. \; loc.}}{\rightrightarrows} f \; \mathrm{on}\; \Omega$.   
\end{definition}

We consider the relationship between uniform convergence with respect to the spherical distance and usual uniform convergence. 
By their definitions, we see that for a set $E \subset \C$
\begin{align*}
f_n \rightrightarrows f \; \mathrm{on}\; E \quad \Longrightarrow \quad f_n \stackrel{\mathrm{sph.}}{\rightrightarrows} f \; \mathrm{on}\; E,
\end{align*}
where $f_n \rightrightarrows f \; \mathrm{on}\; E$ means that $\{ f_n \}_{n=1}^\infty$ converges uniformly  to $f$ on $E$. 
On the other hand, the converse does not necessarily hold (e.g. consider $f_n(z)=e^{z+\frac{1}{n}}$ and the limit function $f(z)=e^z$ on $\Omega=\{ z \in \C : \Re z>0 \}$). 
However, the converse holds if the limit function is bounded.

\begin{fact}{\cite[Theorem 1.2.2]{Sc}}\label{fact:conv}
If a sequence $\{ f_n \}_{n=1}^\infty$ converges spherically uniformly to a bounded function $f$ on $E \subset \C$, then $\{ f_n \}_{n=1}^\infty$ converges uniformly to $f$ on $E$.
\end{fact}

\begin{definition} \label{def:normal}
Let $\F$ be a family of meromorphic functions on a domain $\Omega$ in $\C$.
We say that the family $\F$ is \textit{normal} on $\Omega$ if every sequence $\{ f_n \}_{n=1}^\infty \subset \F$ contains a subsequence $\{ f_{n_k} \}_{k=1}^\infty$ that converges locally uniformly on $\Omega$ with respect to the spherical distance. 
\end{definition}

\begin{remark} \label{rem:normality of mero}
Normality is a notion of compactness:
a family $\F$ of meromorphic functions is normal on a domain $\Omega$ if and only if $\F$ is relatively compact in the topology of spherically uniformly convergence on compact subsets of $\Omega$.
Normality of families of meromorphic functions also has a local property:
a family $\F$ of meromorphic functions is normal on a domain $\Omega$ if and only if, for each point $p \in \Omega$, there exists a neighborhood $U_p$ of $p$ such that $\F$ is normal on $U_p$.   
\end{remark}

\begin{fact}{\cite[Theorem 1.1]{CTC}}\label{fact:conv of mero}
Let $\Omega$ be a domain in $\C$, $a \in \Omega$ and $\{ f_n \}_{n=1}^\infty$ be a sequence of meromorphic functions on $\Omega$.
Assume that the sequence $\{ f_n \}_{n=1}^\infty$ converges locally uniformly to f on $\Omega$ with respect to the spherical distance.
Then the following assertions hold. 

\begin{enumerate}
\item[(\hspace{.18em}i\hspace{.18em})] If $f(a) \ne \infty$, then there exist $r>0$ and $N \in \Z_{>0}$ such that $f_n\;(n \geq N)$ and $f$ are bounded holomorphic functions on $ D(a;r) := \{ z \in \Omega : |z-a| <r \}$ and $f_n \rightrightarrows f \; \mathrm{on} \; D(a;r)$.

\item[(\hspace{.08em}ii\hspace{.08em})] If $f(a)=\infty$, then there exist $r>0$ and $N \in \Z_{>0}$ such that $1/f_n \; (n \geq N)$ and $1/f$ are bounded holomorphic functions on $D(a;r) (\subset \Omega)$ and $1/f_n \rightrightarrows 1/f \; \mathrm{on} \; D(a;r)$.
\end{enumerate}
\end{fact}

\subsection{Compact property} \label{sbsec:2.2}\par

Let $\Sigma$ be a connected Riemann surface. We define  
\begin{align*}
\M(\Sigma) 
& := \{ f : \Sigma \longrightarrow \RC \, : \, \mathrm{holomorphic \; map} \} \\
& = \{ f : \Sigma \longrightarrow \RC \, : \, \mathrm{meromorphic \; function} \} \cup \{ f \equiv \infty \; \mathrm{on} \; \Sigma \}.
\end{align*}
On $\M(\Sigma)$, we consider the topology of the uniform convergence on compact subsets.
Remark that the topology of uniform convergence on compact subsets is equivalent to the compact open topology because $(\RC,\chi)$ is a metric space. Then   
\begin{align*}
f_n \; \mathrm{converges \; to} \; f \; \mathrm{in} \; \M(\Sigma) 
& \Longleftrightarrow ^\forall K \subset \Sigma:\mathrm{compact \; subset}, \lim_{n \rightarrow \infty} \sup_{z \in K} \{ \chi(f_n(z),f(z)) \} =0 \\
& \Longleftrightarrow f_n \stackrel{\mathrm{sph. \; loc.}}{\rightrightarrows} f \quad \mathrm{on} \; \Sigma.
\end{align*}

We also note that for $\RC$-valued holomorphic maps on a Riemann surface $\Sigma$, the notion of normality can be defined as in Definition \ref{def:normal} and all claims in Section \ref{sbsec:2.1} hold as well.
 
Let $P$ be an arbitrary property for $\RC$-valued holomorphic maps and we set
\begin{align*}
\P(\Sigma):=\{ f \in \M(\Sigma) \;:\; f \;\mathrm{satisfies \; the \; property}\; P  \;\mathrm{on}\; \Omega \}.
\end{align*}
For a property $P$, we give the following definitions based on \cite{RA}:

\begin{definition} \label{def:cpt prop}
Given a property $P$, we consider the following assertions: 
\begin{enumerate}
\item[(P$1$)] For any two Riemann surfaces $\Sigma$ and $\Sigma'$, and for any holomorphic map without ramification points $\phi:\Sigma \longrightarrow \Sigma'$, if $f \in \P(\Sigma')$, then $f\circ\phi \in \P(\Sigma)$. 
\item[(P$2$)] Let $\Sigma$ be a Riemann surface and $f \in \M(\Sigma)$. 
If we have $f|_{\Omega} \in \P(\Omega)$ for any relatively compact domain $\Omega$ of $\Sigma$, then $f \in \P(\Sigma)$.
\item[(P$3$)] For any Riemann surface $\Sigma$, the family $\P(\Sigma)$ is a closed subset of $\M(\Sigma)$.
\item[(P$4$)] For any Riemann surface $\Sigma$, the family $\P(\Sigma)$ is normal on $\Sigma$. 
\end{enumerate}
If $P$ satisfies the axioms (P$1$), (P$2$) and (P$3$), $P$ is called a \textit{closed property}. Furthermore, if $P$ satisfies the axioms (P$1$), (P$2$), (P$3$) and (P$4$), $P$ is called a \textit{compact property}.  
\end{definition}

For a holomorphic map $f \colon \Sigma \longrightarrow \RC$ and the stereographic projection $\pi_N \colon \RC \longrightarrow S^2$, we consider 
\begin{align*}
\pi_N \circ f:= \left\{ \,
\begin{aligned}
& \left( \dfrac{2 \, \Re{f}}{|f|^2+1},\dfrac{2 \, \Im{f}}{|f|^2+1},\dfrac{|f|^2-1}{|f|^2+1} \right) & & \mathrm{on} \quad \Sigma \setminus f^{-1}(\{ \infty \}), \\
& (0,0,1) & & \mathrm{on} \quad f^{-1}(\{ \infty \}).
\end{aligned}
\right.
\end{align*}

Let $\Omega$ be a domain in $\C$. Then we denote by $|\nabla f|_e$ the length of the Euclidean gradient of $\pi_N \circ f$ on $\Omega$. 
This corresponds to the spherical derivative in complex analysis (\cite[Section 1.3]{WB},\cite[Lemma 1.6]{CTC},\cite[Section 1.2]{Sc}). 
In fact, for a meromorphic function $f\colon \Omega \to \RC$, we obtain 
\begin{align*}
|\nabla f|_e^2=\dfrac{8\,|f'|^2}{(1+|f|^2)^2}, \; \mathrm{i.e.}, \; |\nabla f|_e=\dfrac{2\sqrt{2}\,|f'|}{1+|f|^2}.
\end{align*} 
and its spherical derivative is given by $f^{\#}=|f'|/(1+|f|^2)$. 
Remark that we have $|\nabla f|_e \equiv 0$ on $\Omega$ if $f \equiv \infty$ on $\Omega$.  
By comparing $|\nabla f|_e$ with $f^{\#}$, we obtain the following criterion which is called the Marty theorem.

\begin{fact}{\cite[Theorem 1.6]{CTC},\cite[Section 3.3]{Sc}}\label{fac:Marty}
Let $\Omega$ be a domain in $\C$ and $\F \subset \M(\Omega)$.Then $\F$ is normal on $\Omega$ if and only if the family 
\[
| \nabla \F |_{e} := \{ |\nabla f|_e \,:\, f \in \F \}
\] is locally bounded on $\Omega$.
\end{fact}

By applying the uniformization theorem to a Riemann surface $\Sigma$, there exists a complete metric compatible with its Riemann structure and constant Gaussian curvature equal to $1$, $0$ or $-1$ on $\Sigma$. 
We will call this metric the \textit{canonical metric} of $\Sigma$ and denote it by $ds^2_c$. In particular, we have $ds^2_c=|dz|^2$ for the complex plane $\C$ and
\begin{align} \label{eq:P metc}
ds^2_c = \left(\dfrac{2}{1-|z|^2}\right)^2 \, |dz|^2 
\end{align}
for the unit disk $\D:=D(0;1)$. For any map $f \in \M(\Sigma)$, we define $|\nabla f|_c$ as the length of the gradient of $\pi_N \circ f$ with respect to $ds^2_c$. Then $|\nabla f|_c$ with respect to the Poincar\'{e} metric \eqref{eq:P metc} is given as follows:
\begin{align*}
|\nabla f|_c=\dfrac{1-|z|^2}{2} \, |\nabla f|_e.
\end{align*}

The following proposition represents a correspondence from normal family theory to value distribution theory.

\begin{proposition} \label{lem:cpt and const}
If $P$ is a compact property, then $\P(\C)$ contains only constant maps.
\end{proposition}

\begin{proof}
Let $f \in \P(\C)$ and $z \in \C$. Since $f \equiv \infty$ on $\C$ is constant, we may assume that $f$ is a meromorphic function on $\C$. 
For each $n \in \Z_{>0}$, we define $\phi_n : \C \longrightarrow \C \; ; \; \phi_n(w) := z + n \, w$ and $f_n : \C \longrightarrow \RC \; ; \; f_n(w) := (f \circ \phi_n)(w)$. Remark that
\begin{align*} 
|\nabla f_n|_e(w)
=\dfrac{2\sqrt{2} \, |f'(\phi_n(w))|}{1+|f(\phi_n(w))|^2} \, |\phi_n'(w)|
=n \, |\nabla f|_e(\phi_n(w)), \;\,
\mathrm{i.e.,} \;\, |\nabla f_n|_e(0) = n \, |\nabla f|_e(z).
\end{align*}
We have $f_n=f \circ \phi_n \in \P(\C)$ for all $n \in \Z_{>0}$ because each holomorphic map $\phi_n:\C \longrightarrow \C$ does not have any ramification points and $P$ satisfies (P$1$).  From (P$4$), $\{ f_n \}_{n=1}^\infty$ is normal on $\C$. By Fact \ref{fac:Marty}, $\{ |\nabla f_n|_e \}_{n=1}^\infty$ is locally bounded on $\C$. In particular, $\{ |\nabla f_n |_e \}_{n=1}^\infty$ is uniformly bounded at $w=0$. We thus have 
\begin{align*}
^\exists M > 0 \quad \mathrm{s.t.} \quad n \, |\nabla f|_e(z)=|\nabla f_n|_e(0) < M \quad(^\forall n \in \Z_{>0}).
\end{align*}
Remark that $M$ does not depend on $n$.
Hence we have $|\nabla f|_e(z)=0$, i.e., $f'(z)=0$.
Since $z$ is an arbitrary point in $\C$, $f'$ vanishes on $\C$, that is, $f$ is constant on $\C$.
\end{proof}

\begin{remark} \label{rem:cpt normal}
We consider the relationship between the normal property and the compact property. Here, a property $P$ for meromorphic functions on a domain in $\C$ is \textit{normal} if the following four conditions are satisfied: 
\begin{enumerate}
\item[(N$1$)] For a domain $\Omega$ in $\C$, if $f \in \P(\Omega)$ and domain $\Omega'\subset \Omega$, then $f|_{\Omega'} \in \P(\Omega')$.    
\item[(N$2$)] For a domain $\Omega$ in $\C$, if $f \in \P(\Omega)$ and $\varphi(z)=\rho z + c$, where $\rho \in \C \setminus \{ 0 \}$ and $c \in \C$, then $f \circ \varphi \in \P(\varphi^{-1}(\Omega))$. 
\item[(N$3$)] Let $\{ \Omega_n \}_{n=1}^\infty$ be a sequence of domains in $\C$ satisfying $\Omega_1 \subset \Omega_2 \subset \cdots \subset \Omega_n \subset \cdots$ and $\bigcup_{n=1}^\infty \, \Omega_n=\C$. If $f_n \in \P(\Omega_n)$ for all $n \in \Z_{>0}$ and $f_n \stackrel{\mathrm{sph.\;loc.}}{\rightrightarrows} f \; \mathrm{on} \; \C$, then $f \in \P(\C)$. 
\item[(N$4$)] If $f \in \P(\C)$, then $f \equiv$ (constant) on $\C$. 
\end{enumerate}

We verify that a closed property $P$ satisfies the above conditions (N$1$), (N$2$) and (N$3$). 
For (N$1$), let $\Omega$ and $\Omega' \subset \Omega$ be  domains, $\iota_{\Omega'}:\Omega' \longrightarrow \Omega$ be the inclusion map and $f \in \P(\Omega)$.
Since $\iota_{\Omega'}$ is a holomorphic function without ramification points and $P$ satisfies (P$1$), we have $f|_{\Omega'}=f\circ \iota_{\Omega'} \in \P(\Omega')$. 
The condition (N$2$) obviously holds by (P$1$).
For (N$3$), let $\Omega_n \nearrow \C$, $f_n \in \P(D_n)\;(^{\forall} n \in \Z_{>0})$ and $f_n \stackrel{\mathrm{sph.\;loc.}}{\rightrightarrows} f \; \mathrm{on} \; \C$.
Let $U$ be an arbitrary relatively compact domain in $\C$.
By the compactness of $\overline{U}$ and the monotonically increasing of $\{ \Omega_n \}_{n=1}^{\infty}$, there exists $N \in \Z_{>0}$ such that $U \subset \overline{U} \subset \Omega_N$.
Since a closed property $P$ satisfies the above condition (N$1$), $\{ f_n|_U \}_{n=N}^\infty \subset \P(U)$.
Due to $f_n|_{U} \stackrel{\mathrm{sph.\;loc.}}{\rightrightarrows} f|_{U}$ and (P$3$), $f|_{U} \in \P(U)$.
Hence we obtain $f \in \P(\C)$ by (P$2$). 

From this result and Proposition \ref{lem:cpt and const}, a compact property is also normal. 
However, a normal property is not necessarily a compact property.
This is because a compact property $P$ satisfies the assertion
\begin{center}
If $f \in \P(\C \setminus \{ 0 \})$, then $f \circ \exp \in \P(\C)$,
\end{center}
but a normal property does not impose the above condition.
Remark that this condition is essential for the Minda principle (See \cite[Section 1.8]{WB}, \cite{Mi} for details) and the following proposition. \par
\end{remark}

\begin{proposition}\label{prop:meromorphically extension}
Let $P$ be a compact property.
If $f \in \P(\D \setminus \{ 0 \})$, then $z=0$ is not an essential singularity of $f$, that is, $f$ is a holomorphic map from $\D$ to $\RC$.
\end{proposition}
\begin{proof}
We here give a proof based on \cite[Proof of Theorem 1.8.3]{WB}.
It is proved by contradiction. Suppose that there exists a holomorphic map $f:\D \setminus \{ 0 \} \longrightarrow \RC$ such that $z=0$ is an essential singularity of $f$ and $f \in \P(\D \setminus \{ 0 \})$. 
By the Lehto--Virtanen theorem (\cite[Theorem 1.8.4]{WB}, \cite{LV}), $\limsup_{z \rightarrow 0}\{ |z|\,|\nabla f|_e(z) \} \geq \sqrt{2}$, that is, there exists a sequence $\{ c_n \}_{n=1}^\infty$ in $\D$ such that $c_n \rightarrow 0$ and $|c_n|\,|\nabla f|_e(c_n) \geq \sqrt{2}/2$ for all $n \in \Z_{>0}$.
We define $r_n:= 1/(2\,|c_n|)$ and $g_n : D(0;r_n) \setminus \{ 0 \} \longrightarrow \RC \; ; \; g_n(\zeta) := f(c_n \, \zeta)$.
By (P$1$), $g_n \in \P(D(0;r_n) \setminus \{ 0 \})$ for all $n \in \Z_{>0}$. 
Moreover, $\zeta=0$ is an essential singularity of $g_n$. 

Then, by taking a subsequence if necessary, we find that there exists a nonconstant meromorphic function $g:\C \setminus \{ 0 \} \longrightarrow \RC$ such that $g_n \stackrel{\mathrm{sph.\;loc.}}{\rightrightarrows} g \; \mathrm{on} \;\C \setminus \{ 0 \}$ and $g \in \P(\C \setminus \{ 0 \})$. 
We now prove it.
Take a fixed $R>1$.
By $r_n \rightarrow \infty$ (as $n \rightarrow \infty$), there exists $N \in \Z_{>0}$ such that $r_n >2R$ for all $n \geq N$. 
From Remark \ref{rem:cpt normal} (N$1$), $\{ g_n|_{D(0;2R) \setminus \{ 0 \}} \}_{n=N}^\infty \subset \P(D(0;2R) \setminus \{ 0 \})$. By (P$4$), $\P(D(0;2R) \setminus \{ 0 \})$ is normal on $D(0;2R) \setminus \{ 0 \}$. 
Hence, by taking a subsequence if necessary, there exists $g:\overline{D(0;R)} \setminus \{ 0 \} \longrightarrow \RC$ such that $g_n \stackrel{\mathrm{sph.}}{\rightrightarrows} g \; \mathrm{on} \;\overline{D(0;R)} \setminus \{ 0 \}$. 
Here, $g$ is nonconstant  on $\overline{D(0;R)} \setminus \{ 0 \}$. 
This is because
\begin{align*}
|\nabla g_n|_e(1)
= \dfrac{2\sqrt{2} \, |f'(c_n)|}{1+|f(c_n)|^2} \, |c_n|
= |c_n| \, |\nabla f|_e(c_n)
\geq \dfrac{\sqrt{2}}{2}
\end{align*}
for all $n \in \Z_{\geq N}$. 
From (P$3$), $g|_{D(0;R) \setminus \{ 0 \}} \in \P(D(0;R) \setminus \{ 0 \})$. 
Since $R$ is arbitrary and $P$ satisfies (P$2$), $g_n \stackrel{\mathrm{sph.\;loc.}}{\rightrightarrows} g \; \mathrm{on} \;\C \setminus \{ 0 \}$ and $g \in \P(\C \setminus \{ 0 \})$. 

From (P$1$), $g \circ \exp \in \P(\C)$. 
By Proposition \ref{lem:cpt and const}, $g$ is constant on $\C \setminus \{ 0 \}$. 
This is a contradiction.
\end{proof}

\begin{lemma}\label{lem:cpt=normal on D}
Let $P$ be a closed property. The followings are equivalent: 
\begin{enumerate}
\item[$(a)$] $P$ is a compact property.  
\item[$(b)$] $\P(\D)$ is normal on $\D$.
\end{enumerate} 
\end{lemma}

\begin{proof}
By (P$4$), $(a)$ implies obviously $(b)$. 
We prove that $(b)$ implies $(a)$. 
Let $\Sigma$ be a Riemann surface and $p \in \Sigma$. 
Then we take a local coordinate $(\Omega_p, \phi_p)$ of $p$ such that $\phi_p : \Omega_p \longrightarrow \D$ is a homeomorphism. 
Let $\{ f_n \}_{n=1}^\infty$ be an arbitrary sequence of $\P(\Sigma)$. 
By Remark \ref{rem:cpt normal} (N$1$), $\{ f_n|_{\Omega_p} \}_{n=1}^\infty \subset \P(\Omega_p)$. 
In addition, $\{ g_n := f_n|_{\Omega_p} \circ \phi_p^{-1} \}_{n=1}^\infty \subset \P(\D)$ from (P$1$). 
Since $\P(\D)$ is normal on $\D$, there exists a subsequence $\{ g_{n_k} \}_{k=1}^\infty$ of $\{ g_n \}_{n=1}^\infty$ and $g:\D \longrightarrow \RC$ such that $g_{n_k} \stackrel{\mathrm{sph.\;loc.}}{\rightrightarrows} g \; \mathrm{on} \; \D$. 
Hence a subsequence $\{ f_{n_k} \}_{k=1}^\infty$ of $\{ f_n \}_{n=1}^\infty$ satisfies $f_{n_k} \stackrel{\mathrm{sph.\;loc.}}{\rightrightarrows} g \circ \phi_p \; \mathrm{on} \; \Omega_p$. 
Therefore $\P(\Sigma)$ is normal on a neighborhood of each point in $\Sigma$. 
By Remark \ref{rem:normality of mero}, $\P(\Sigma)$ is normal on $\Sigma$, that is, $P$ is a compact property.     
\end{proof}

The following theorem corresponds to the Zalcman lemma (for example, see \cite[Chapter 2]{LZ}) in complex analysis. 
This result is often used if we apply results in value distribution theory to normal family theory.

\begin{theorem}\label{thm:Zalcman}
Let $P$ be a closed property. 
Then $\rm{(\hspace{.18em}i\hspace{.18em})}$ or $\rm{(\hspace{.08em}ii\hspace{.08em})}$ holds: 
\begin{enumerate}
\item[\rm(\hspace{.18em}i\hspace{.18em})] $P$ is a compact property. 
\item[\rm(\hspace{.08em}ii\hspace{.08em})] There exists $f \in \P(\C)$ such that $|\nabla f|_e(0)=1$ and $|\nabla f|_e \leq 1$ on $\C$. 
\end{enumerate}
\end{theorem}

\begin{proof}
Suppose that $P$ is not compact but closed. 
By Lemma \ref{lem:cpt=normal on D}, $\P(\D)$ is not normal on $\D$. 
Due to Fact \ref{fac:Marty}, the family $|\nabla \P (\D)|_{e} = \{ |\nabla f|_e : f \in \P(\D) \}$ is not locally bounded on $\D$, that is, 
\begin{align*}
^\exists r_0 \in (0,1), \, ^\exists z_n \in \overline{D(0;r_0)}, \, ^\exists g_n \in \P(\D) \;\, \mathrm{s.t.} \;\, |\nabla g_n|_e(z_n) \rightarrow \infty, \, |\nabla g_n|_e(z_n) >0 \; (^\forall n \in \Z_{>0}).
\end{align*}
From the compactness of $\overline{D(0;r_0)}$, by taking a subsequence if necessary, there exists $z_0 \in \overline{D(0;r_0)}$ such that $z_n \rightarrow z_0$. 
Let $ds^2_c$ be the Poincar\'{e} metric \eqref{eq:P metc}. 
Then 
\begin{align*}
|\nabla g_n|_c(z_n)=\dfrac{1-|z_n|^2}{2}\,|\nabla g_n|_e(z_n) \rightarrow \infty.
\end{align*}
Since $\D$ is homogeneous, we can choose the conformal transformation $\phi_n:\D \longrightarrow \D\;;$
\begin{align*}
\phi_n(z):=\dfrac{z-z_n}{z\,\overline{z_n}-1}, 
\end{align*}
which maps $0$ onto $z_n$ for each $n \in \Z_{>0}$. 
By (P$1$), $g_n \circ \phi_n \in \P(\D)$. We also define the nonconstant meromorphic functions $h_n : \D \longrightarrow \RC \; ; \; h_n(z) := g_n \left(\phi_n\left( z/2 \right) \right)$. 
Again from (P$1$), $h_n \in \P(\D)$ and we obtain
\begin{align*}
|\nabla h_n|_c(0)
= \dfrac{1-0^2}{2} \, |\nabla h_n|_e(0) 
= \dfrac{1}{2} \, \dfrac{1-|z_n|^2}{2} \, |\nabla g_n|_e(z_n) 
= \dfrac{1}{2} \, |\nabla g_n|_c(z_n) \rightarrow \infty.
\end{align*}
Furthermore, $|\nabla h_n|_c \equiv 0$ on $\partial \D$ because $|\nabla h_n|_c(z) =((1-|z|^2)/2)\,|\nabla h_n|_e(z)$ and 
\begin{align*}
|\nabla h_n|_e(z)
= \dfrac{2\sqrt{2} \, \left|g'_n\left(\phi_n\left( z/2 \right)\right)\right|}{1+\left|g_n\left(\phi_n\left( z/2 \right)\right)\right|^2} \, \left|\phi'_n\left( z/2 \right)\right| \, \dfrac{1}{2}
= \dfrac{1}{2} \, |\nabla g_n|_e \left(\phi_n\left( z/2 \right)\right) \, \dfrac{1-|z_n|^2}{\left|\overline{z_n} \, (z/2)-1\right|^2}
\end{align*}
exists as a real number on $\partial \D$. 
Remark that each $|\nabla h_n|_c$ has a maximum value on $\overline{\D}$ due to the continuity of $|\nabla h_n|_c$ on $\overline{\D}$. 
Hence we see that the point which takes a maximum value of $|\nabla h_n|_c$ is an interior point of $\D$. 
By composing with conformal transformations of $\D$ if necessary, we can assume that, for each $n \in \Z_{>0}$, 
\begin{align*}
\max_{z \in \D} |\nabla h_n|_c(z) = |\nabla h_n|_c(0).
\end{align*}
Set $R_n := 2 \, |\nabla h_n|_c(0)$. Then we have
\begin{align} \label{eq:nabla h_n(0)}
|\nabla h_n|_e(0) 
= \dfrac{2}{1-0^2} \, |\nabla h_n|_c(0) 
= R_n 
\rightarrow \infty,
\end{align}
\begin{align} \label{eq:nabla h_n(z)}
|\nabla h_n|_e(z) 
= \dfrac{2}{1-|z|^2} \, |\nabla h_n|_c(z) 
\leq \dfrac{2}{1-|z|^2} \, |\nabla h_n|_c(0) 
= \dfrac{R_n}{1-|z|^2} \;\, (^\forall z \in \D, \, ^\forall n \in \Z_{>0}).
\end{align}
Moreover, we set the sequence of nonconstant meromorphic functions $f_n: D(0;R_n) \longrightarrow \RC \; ; \; f_n(z) := h_n\left( z/R_n \right)$. 
From (P$1$), $f_n \in \P(D(0;R_n))$. 
By using \eqref{eq:nabla h_n(0)} and \eqref{eq:nabla h_n(z)}, for each $n \in \Z_{>0}$,
\begin{align} \label{eq:nabla f_n(0)}
|\nabla f_n|_e(0)
= \dfrac{2\sqrt{2} \, |h'_n(0)|}{1+|h_n(0)|^2} \, \dfrac{1}{R_n}
= \dfrac{1}{R_n} \, |\nabla h_n|_e(0)
= 1,
\end{align}
\begin{align} \label{eq:nabla f_n(z)}
|\nabla f_n|_e(z)
= \dfrac{2\sqrt{2} \, \left|h'_n \left( z/R_n \right)\right|}{1+\left|h_n \left( z/R_n \right)\right|^2} \, \dfrac{1}{R_n}
= \dfrac{1}{R_n}\,|\nabla h_n|_e \left( z/R_n \right) \leq \dfrac{1}{1-\left( |z|/R_n \right)^2} \quad \mathrm{on} \; D(0;R_n).
\end{align}
So we take a fixed $R>0$. 
By $R_n \rightarrow \infty$, there exists $N \in \Z_{>0}$ such that $R_n > 2R$ for all $n \geq N$. 
In particular, $f_n\;(n \geq N)$ are defined on $\overline{D(0;2R)}( \subset D(0;R_n))$. 
By \eqref{eq:nabla f_n(z)}, for all $z \in D(0;2R)$ and for all $n \geq N$, we have
\begin{align*}
|\nabla f_n|_e(z) \leq \dfrac{1}{1-\left( |z|/R_n \right)^2} < \dfrac{1}{1-\left( 2R/R_n \right)^2}.
\end{align*}
Since $(1-\left( 2R/R_n \right)^2)^{-1} \rightarrow 1$ (as $n \rightarrow \infty$), the family $\{ |  \nabla f_n |_e \}_{n=N}^\infty$ is uniformly bounded on $D(0;2R)$. 
From Fact \ref{fac:Marty}, $\{ f_n \}_{n=N}^\infty$ is normal on $D(0;2R)$. 
By taking a subsequence if necessary, there exists $f:\overline{D(0;R)} \longrightarrow \RC$ such that $f_n \stackrel{\mathrm{sph.}}{\rightrightarrows} f \, \mathrm{on}\,  \overline{D(0;R)}$. 
Since $R$ is arbitrary, it implies $f_n \stackrel{\mathrm{sph.\;loc.}}{\rightrightarrows} f \; \mathrm{on} \;\C$. 
Furthermore, we see that the limit function $f$ is a nonconstant meromorphic function on $\C$ by \eqref{eq:nabla f_n(0)}. 
From Remark \ref{rem:cpt normal} (N$1$), we have $f_n|_{D(0;R)} \in \P(D(0;R))$ for all $n \geq N$. 
By applying (P$3$) to $\{ f_n|_{D(0;R)} \}_{n=N}^\infty (\subset \P(D(0;R)))$, $f|_{D(0;R)} \in \P(D(0;R))$. 
From (P$2$), $f \in \P(\C)$. 
Due to \eqref{eq:nabla f_n(0)} and \eqref{eq:nabla f_n(z)}, we obtain
$|\nabla f|_e(0)=1$ and  $|\nabla f|_e(z) \leq 1$ for any $z \in \C$.
\end{proof}

We give two examples of compact property. 
The first one corresponds to a weaker result of the Montel theorem and is proved by applying Theorem \ref{thm:Zalcman} and the Liouville theorem.  

\begin{proposition} \label{prop:Liouville Montel}
Let $L$ be a positive number. 
For a Riemann surface $\Sigma$ and a holomorphic map $f:\Sigma \longrightarrow \RC$, we define the property $P_L$ as 
\begin{center}
$|f| < L$ on $\Sigma$ \; or \; $f\equiv$ $($constant$)$ on $\Sigma$.
\end{center}
Then $P_L$ is a compact property.
\end{proposition}
\begin{proof}
We first verify that $P_L$ is a closed property. 

(P$1$) Let $\Sigma$ and $\Sigma'$ be Riemann surfaces, $\phi:\Sigma \longrightarrow \Sigma'$ be a holomorphic map without ramification points and $f \in \P_L(\Sigma')$. 
If $|f(q)|<L \;(^\forall q \in \Sigma')$, then we have $|f(\phi(p))|<L\;(^\forall p \in \Sigma)$. 
If $f$ is constant on $\Sigma'$, then $f \circ \phi$ is also constant on $\Sigma$. 
We thereby obtain $f \circ \phi \in \P(\Sigma)$.  

(P$2$) Let $\Sigma$ be a Riemann surface and $f \in \M(\Sigma)$. 
We prove the following contraposition:
\begin{center}
If $f \notin \P_L(\Sigma)$, there exists a relatively compact domain $\Omega$ of $\Sigma$ and $p \in \Omega$ \\ such that $|f|_{\Omega}(p)| \geq L$ and $f|_{\Omega}$ is nonconstant on $\Omega$.
\end{center}
If $f \notin \P_L(\Sigma)$, then $f$ is nonconstant on $\Sigma$ and there exists $p \in \Sigma$ such that $|f(p)| \geq L$. 
We take a neighborhood $\Omega_p$ of $p$ such that $\overline{\Omega_p}$ is a compact subset in $\Sigma$. 
Then we obtain $|f|_{\Omega_p}(p)| \geq L$ and $f|_{\Omega_p}$ is nonconstant on $\Omega_{p}$ by the identity theorem. 

(P$3$) For a Riemann surface $\Sigma$, we consider any sequence $\{ f_n \}_{n=1}^\infty$ in $\P_L(\Sigma)$ satisfying $f_n \stackrel{\mathrm{sph.\;loc.}}{\rightrightarrows} f \; \mathrm{on} \;\Sigma$. 
Assume that the limit function $f$ is nonconstant on $\Sigma$. 
Then we can prove $|f|<L$ on $\Sigma$ by contradiction. 
Suppose that there exists $p \in \Sigma$ such that $|f(p)| \geq L$. 
We set $q=f(p)$. 
By Fact \ref{fact:conv of mero} and the Hurwitz theorem, there exists a neighborhood $\Omega_p$ of $p$ and $N \in \Z_{>0}$ such that $q \in f_n(\Omega_p)\;(n \geq N)$. 
From $f_n \in \P_L(\Sigma)\;(n \geq N)$, each $f_n\;(n \geq N)$ is  constant on $\Sigma$. 
By $f_n \stackrel{\mathrm{sph.\;loc.}}{\rightrightarrows} f \; \mathrm{on} \;\Sigma$, $f$ is constant on $\Sigma$ . 
This is a contradiction. 

Therefore $P_L$ is a closed property. 
We next show that $P_L$ is a compact property. 
Suppose that $P_L$ is not compact. 
By Theorem \ref{thm:Zalcman}, there exists $f \in \P_L(\C)$ satisfying $|\nabla f|_e(0)=1$. 
Remark that $f$ is a nonconstant meromorphic function on $\C$. 
So we obtain $|f|< L$ on $\C$. 
By the Liouville theorem, $f$ is constant on $\C$. 
This is a contradiction.  
\end{proof}

We next give a significant example of compact property. 
Based on \cite[Section 1.7]{WB}, \cite[Theorem 2]{RA} and \cite[Chapter 3]{LZ}, we first prove the Picard little theorem and next obtain the compactness of $P_X$, that is, the Carath\'{e}odory--Montel theorem by applying this theorem and Theorem \ref{thm:Zalcman}.  

\begin{theorem} \label{thm:Picard Montel}
Let $X$ be a subset in $\RC$. 
For a Riemann surface $\Sigma$ and a holomorphic map $f\colon \Sigma \longrightarrow \RC$, we define the property $P_X$ as 
\begin{center}
$f(\Sigma) \subset \RC \setminus X$ \; or \; $f\equiv$ $($constant$)$ on $\Sigma$.
\end{center}
Then $P_X$ is a closed property. 
Furthermore, $P_X$ is a compact property if $X$ contains $3$ or more elements. 
\end{theorem}

\begin{proof}
By the same argument as in Proposition \ref{prop:Liouville Montel}, we see that $P_X$ is a closed property. 
So we will show that $P_X \; (|X| = 3)$ is a compact property in three steps. 

We first prove that $P_X$ is compact if $X=\{ 0, \, 1, \, \infty \}$ by contradiction. 
Suppose that $P_X$ is not compact. 
By Lemma \ref{lem:cpt=normal on D}, $\P_X(\D)$ is not normal on $\D$. 
Since any holomorphic map $f\colon \D \longrightarrow \RC$ that omits all elements of $X$ is a holomorphic function without zeros on the simply connected domain $\D$, for each $n \in \Z_{>0}$, there exists a holomorphic function $g: \D \longrightarrow \C$ such that $g^{2^n}=f$ on $\D$. 
So we set 
\begin{align*}
\F_n 
:= \left\{ g \; : \; ^\exists f \in \M(\D) \quad \mathrm{s.t.} \quad f(\D) \subset \RC \setminus X \; \mathrm{and} \; g^{2^n}=f \;\mathrm{on}\;\D \right\}
\end{align*} 
for each $n \in \Z_{>0}$. 
For any $f, \, g$ satisfying $g^{2^n} =f$ on $\D$, by  
\[
|g'(z)|
= \dfrac{1}{2^n} \, \dfrac{|g(z)|}{|g^{2^n}(z)|} \, |f'(z)|
= \dfrac{1}{2^n} \, \dfrac{|f(z)|^{\frac{1}{2^n}}}{|f(z)|} \, |f'(z)|
\] 
and applying the inequality $a^x+a^{-x} \leq a+a^{-1} \; (a>0, \, 0<x<1)$, we obtain 
\begin{align*}
|\nabla g|_e(z)
= \dfrac{2\sqrt{2}}{2^n} \, \dfrac{(|f(z)|^\frac{1}{2^n}/|f(z)|) \, |f'(z)|}{1+|f(z)|^\frac{2}{2^n}} 
& = \dfrac{2\sqrt{2}}{2^n} \, \dfrac{|f(z)|^{\frac{1}{2^n}}}{1+|f(z)|^{\frac{2}{2^n}}} \, \dfrac{|f'(z)|}{|f(z)|} \, \dfrac{1+|f(z)|^2}{1+|f(z)|^2} \\
& = \dfrac{1}{2^n} \, \dfrac{2 \sqrt{2} \, |f'(z)|}{1+|f(z)|^2} \, \dfrac{|f(z)|+|f(z)|^{-1}}{|f(z)|^{\frac{1}{2^n}}+(|f(z)|^{\frac{1}{2^n}})^{-1}}\\
& \geq \dfrac{|\nabla f|_e(z)}{2^n},
\end{align*}
that is, $|\nabla g|_e(z) \geq |\nabla f|_e(z)/2^n$. 
From Fact \ref{fac:Marty}, the family $|\nabla \P_X(\D)|_e$ is not locally bounded on $\D$. 
We thereby see that, for each $n \in \Z_{>0}$, the family $|\nabla \F_n|_e$ is also not locally bounded on $\D$, that is, $\F_n$ is not normal on $\D$. Set 
\begin{align*}
X_n := \left\{ 0, \, \infty, \, 1, \, e^{\frac{2\pi i}{2^n}}, \, \cdots, \, e^{\frac{2(2^n-1)\pi i}{2^n}} \right\}. 
\end{align*}
By the definition of $\F_n$ and $\P_X(\D)$, we have $\F_n \subset \P_{X_n}(\D)$. 
Thus each $P_{X_n}$ is not compact but closed. 
By Theorem \ref{thm:Zalcman}, for each $n \in \Z_{>0}$, there exists $g_n \in \P_{X_n}(\C)$ such that 
\begin{align*}
|\nabla g_n|_e(0)=1 \quad \mathrm{and} \quad |\nabla g_n|_e(z) \leq 1 \;\,(^\forall z \in \C).
\end{align*}
Then $\G := \{ g_n \}_{n=1}^\infty$ is a sequence of meromorphic functions and $|\nabla \G|_e$ is uniformly bounded on $\C$. 
From Fact \ref{fac:Marty}, $\G$ is normal on $\C$. 
Due to $|\nabla g_n|_e(0)=1$, by taking a subsequence if necessary, there exists a nonconstant meromorphic function $h:\C \longrightarrow \RC$ such that $g_n \stackrel{\mathrm{sph.\;loc.}}{\rightrightarrows} h \; \mathrm{on} \;\C$. 
By the monotonically increasing of $X_n$ and (P$3$), $h$ omits all elements of $\bigcup_{n=1}^\infty X_n$. 
Since $\bigcup_{n=1}^\infty X_n \setminus \{ 0,\infty \}$ is dense in $S^1$ and the image $h(\C)$ is an open connected set, we obtain
\begin{align*}
h(\C) \subset \D \quad \mathrm{or} \quad h(\C) \subset \C \setminus \overline{\D}.
\end{align*}
In either case, $h$ is constant on $\C$ by the Liouville theorem. 
This is a contradiction.

We next prove the Picard little theorem. 
Let $f : \C \longrightarrow \RC$ be a meromorphic function that omits distinct values $a, \, b, \, c \in \RC$. 
Since there exists the linear fractional transformation $T : \RC \longrightarrow \RC$ such that $T(a) = 0$, $T(b) = 1$, $T(c) = \infty$,
$T \circ f$ is a meromorphic function that omits all elements of $X := \{ 0, \, 1, \, \infty \}$, that is, $T \circ f \in \P_X(\C)$. 
By Proposition \ref{lem:cpt and const}, $T \circ f$ is constant, that is, $f$ is constant on $\C$.

By applying the Picard little theorem and Theorem \ref{thm:Zalcman}, we can prove that $P_X$ is a compact property in general case as in Proposition \ref{prop:Liouville Montel}.

\end{proof}

\remark \label{rem:Picard great thm} By applying Proposition \ref{prop:meromorphically extension} to the compact property $P_X \; (|X| \geq 3)$, we also obtain the contraposition of the Picard great theorem. 
More precisely, if a holomorphic map $f : \D \setminus \{ 0 \} \longrightarrow \RC$ omits $3$ or more points, then the puncture $z = 0$ is not an essential singularity of $f$.

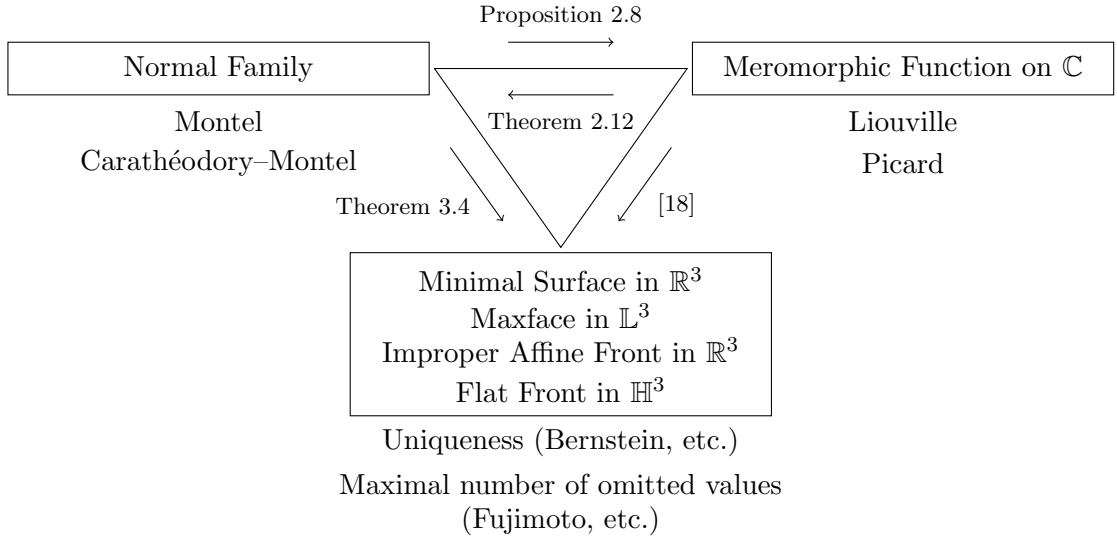
\section{Generalization of Bloch--Ros principle} \label{sec:3}

In this section, we formulate a necessary condition for a compact property not to satisfy a certain curvature estimate, which is a generalization of \cite[Theorem 3]{RA}. 
Furthermore, by applying this criterion to $P_L$ and $P_X$, we obtain the Liouville-type theorem and the Picard-type theorem (\cite[Corollary 2.2]{YK2013}) for an open Riemann surface with a specified conformal metric. 

\subsection{Main results}\label{sbsec:3.1} 

We first give the notion of Weierstrass $m$-triples. 
This is a generalization of a Weierstrass data for minimal surfaces in $\R^{3}$. 

\begin{definition} \label{def:m-pair}
Let $\Sigma$ be an open Riemann surface with the conformal metric
\begin{align} \label{eq:m-metric}
ds^2 := (1+|g|^2)^m \, |f|^2 \, |dz|^2,
\end{align}
where $f \, dz$ is a holomorphic $1$-form on $\Sigma$, $g$ is a meromorphic function on $\Sigma$ and $m \in \Z_{>0}$. Then we call this triple $(\Sigma, \, f \, dz, \, g)$ a \textit{Weierstrass $m$-triple}.   
\end{definition}

\begin{remark}
We note that
\begin{align} \label{eq:regularity}
& 0 < (1+|g|^2)^m \, |f|^2 < +\infty \quad \Longleftrightarrow \quad (f \, dz)_0 = m \, (g)_\infty,  
\end{align}
where $(f \, dz)_{0}$ is the zero divisor of $f \, dz$ and $(g)_{\infty}$ is the polar divisor of $g$. 
Moreover, for a conformal metric $ds^{2}= (\lambda(z))^2 \, |dz|^2$, we have 
$K_{ds^2} =-(\Delta\log{\lambda})/ \lambda^{2}$. 
Hence the Gaussian curvature $K_{ds^2}$ of any Weierstrass $m$-triple $(\Sigma, \, f \, dz, \, g)$ is given by
\begin{align} \label{eq:Gauss curvature of m-Riemann surf.}
K_{ds^2} = - \dfrac{2 \, m \, |g'|^2}{(1+|g|^2)^{m+2} \, |f|^2}.
\end{align}
\end{remark}

\begin{definition} \label{def:m-curvature est}
Let $P$ be a compact property for $\RC$-valued holomorphic maps and $m \in \Z_{>0}$. 
We say that $P$ satisfies \textit{$m$-curvature estimate} if there exists a positive constant $C=C(P,m)>0$ such that for any Weierstrass $m$-triple $(\Sigma, \, f \, dz, \, g)$ satisfying $g \in \P(\Sigma)$, the following inequality holds: 
\begin{align*}
|K_{ds^2}| \, d^2 \leq C \quad\mathrm{on} \; \Sigma.
\end{align*}
Here, $K_{ds^2}$ is the Gaussian curvature with respect to the metric $ds^2$ given by \eqref{eq:m-metric} and $d$ is its geodesic distance to the boundary of $\Sigma$.
\end{definition}

This is one of the main results in this paper and we prove this result in Section \ref{sbseq:3.2}.       

\begin{theorem} \label{thm:m-curvature est}
Let $P$ be a compact property for $\RC$-valued holomorphic maps and $m \in \Z_{>0}$. 
Then $\rm(\hspace{.18em}i\hspace{.18em})$ or $\rm(\hspace{.08em}ii\hspace{.08em})$ holds:

\begin{enumerate}
  \item[\rm(\hspace{.18em}i\hspace{.18em})] $P$ satisfies $m$-curvature estimate. 
  \item[\rm(\hspace{.08em}ii\hspace{.08em})] There exists a Weierstrass $m$-triple $(\D, \, f \, dz, \, g)$ such that 
\begin{enumerate}
  \item[(A)] nonconstant meromorphic function $g \in \P(\D)$,
  \item[(B)] $ds^2 :=(1+|g|^2)^m \, |f|^2 \, |dz|^2$ is complete on $\D$, 
  \item[(C)] $|K_{ds^2}(0)|=\dfrac{1}{4}$ and $|K_{ds^2}| \leq 1$ on $\D$. 
\end{enumerate}  
\end{enumerate}
 \end{theorem}

By applying this result to the compact property $P_L$ and $P_X \; (|X| \geq m + 3)$, we can show that both properties satisfy $m$-curvature estimate.
We first verify that $P_L$ satisfies $m$-curvature estimate.

\begin{theorem} \label{thm:Bernstein}
Let $L$ be a positive number and $P_L$ be the property in Proposition \ref{prop:Liouville Montel}.
Then, for any $m \in \Z_{>0}$, $P_L$ is a compact property that satisfies $m$-curvature estimate.
In particular, for any $L > 0$ and for any $m \in \Z_{>0}$, there exists $C = C(L,m) > 0$ such that for any Weierstrass $m$-triple $(\Sigma, \, f \, dz, \, g)$ satisfying $|g| < L$ on $\Sigma$, the following inequality holds:
\begin{align*}
|K_{ds^2}(p)|^{\frac{1}{2}} \leq \dfrac{C}{d(p)} \quad(^\forall p \in \Sigma).
\end{align*}
\end{theorem}

\begin{proof}
Take an arbitrary $m \in \Z_{>0}$ and fix it. 
Suppose that $P_L$ does not satisfy $m$-curvature estimate. 
By Theorem \ref{thm:m-curvature est}, there exists a Weierstrass $m$-triple $(\D, \, f \, dz, \, g)$ such that $ds^2 = (1 + |g|^2)^m \, |f|^2 \, |dz|^2$ is a complete metric on $\D$ and nonconstant meromorphic function $g \in \P_L(\D)$, that is, $|g| < L$ on $\D$. 
Set $ds_0^2:=(1+L^2)^m \, |f|^2 \, |dz|^2$. 
Then $ds_0^2$ is a complete metric on $\D$ because 
\begin{align*}
(1+|g|^2)^{\frac{m}{2}} \, |f| < (1+L^2)^{\frac{m}{2}} \, |f| \quad \mathrm{on} \; \D.
\end{align*}
Furthermore, $f$ has no zeros on $\D$ by \eqref{eq:regularity}. 
Hence we obtain 
\begin{align*}
\Delta \, \log \left\{ (1+L^2)^{\frac{m}{2}} \, |f| \right\} = 0, \quad \mathrm{i.e.,} \quad K_{ds_0^2} \equiv 0 \quad \mathrm{on} \; \D.
\end{align*}
In other words, $ds_0^2$ is a flat metric on $\D$. 
Since the universal covering of a complete and flat Riemannian $2$-manifold is the plane, the universal covering of $(\D,ds_0^2)$ is $\C$. 
This is a contradiction.
\end{proof}

This result implies the following Liouville-type result.

\begin{corollary} \label{cor:Bernstein}
Let $(\Sigma, \, f \, dz, \, g)$ be a Weierstrass $m$-triple. 
If the metric given by \eqref{eq:m-metric} is complete on $\Sigma$ and $g$ is bounded on $\Sigma$, then $g$ is constant on $\Sigma$.
\end{corollary}

\begin{proof}
Assume that the metric $ds^2$ given by \eqref{eq:m-metric} is complete on $\Sigma$ and there exists $L>0$ such that $|g|<L$ on $\Sigma$. 
By Theorem \ref{thm:Bernstein}, there exists a positive constant $C=C(L,m)>0$ such that $|K_{ds^2}(p)|^{\frac{1}{2}} \leq C/d(p) \;(^\forall p \in \Sigma)$. 
Since $d(p) = \infty$ for any $p \in \Sigma$ in the case where $ds^2$ is a complete metric on $\Sigma$, $K_{ds^2} \equiv 0$ on $\Sigma$, that is, $g$ is constant on $\Sigma$.
\end{proof}

As an application of Theorem \ref{thm:m-curvature est}, we give another proof of \cite[Theorem 2.1]{YK2013}. 
We first recall the following function-theoretic lemma. Remark that this lemma is different from the lemma \cite[Lemma 3.1]{FH1988} that 
was used in \cite[Theorem 4]{RA}.   

\begin{fact}\cite[(8.12) on page 136]{FH1997} \label{fac:Fujimoto lemma}
Let $q \in \Z_{\geq 3}$, $\alpha_1,\cdots,\alpha_{q-1} \in \C$ be distinct, $\alpha_q=\infty$ and $f : D(0;R) \longrightarrow \RC$ be a nonconstant holomorphic function which omits $q$ values $\alpha_1, \cdots, \alpha_q$. Then, for any $\eta$ with $0<\eta<(q-2)/q$, there exists $C >0$ such that 
\begin{align*}
\dfrac{|f'|}{(1+|f|^2) \, \prod_{j=1}^q \chi(f,\alpha_j)^{1-\eta}} \leq C \, \dfrac{R}{R^2-|z|^2} \quad \mathrm{on} \; D(0;R).
\end{align*}
\end{fact}

We also need the following result.

\begin{lemma} \label{lem:Yau-Schwarz}
Let $ds^2 = (\lambda(z))^2 \, |dz|^2$ be a complete Hermitian metric on the unit disk $\D$ whose Gaussian curvature $K_{ds^2}$ satisfies $K_{ds^2} \geq -1$. Then 
\begin{align*}
\lambda(z) \geq \dfrac{2}{1-|z|^2} \quad\mathrm{on}\;\D.
\end{align*}
\end{lemma}

In order to prove this result, we recall the Yau--Schwarz lemma (\cite[Theorem 7.1]{KL}, \cite[Theorem 2]{Yau}), which is dual to the Ahlfors--Schwarz lemma.

\begin{fact} \cite[Theorem 7.1]{KL} \label{fac:Yau-Schwarz}
Let $(M,g)$ be a complete K\"{a}hler manifold with its Ricci curvature bounded from below by a negative constant $-k$ and let $(N,h)$ be a Hermitian manifold with its holomorphic bisectional curvature bounded from above by a negative constant $-K$.
Then every holomorphic map $f : M \longrightarrow N$ satisfies 
\begin{align*}
  f^* \, h \leq \dfrac{k}{K} \, g,
\end{align*}
where $f^* \, h$ denotes the pullback of the metric $h$ by $f$.
\end{fact}

\begin{proof}[\textsc{Proof of lemma \ref{lem:Yau-Schwarz}}]

We apply Fact \ref{fac:Yau-Schwarz} to $(M, g)=(\D, ds^2)$, $(N, h)=(\D, ds_c^2)$ and $f = \mathrm{id} \colon M \longrightarrow N$, where $ds_c^2$ is the Poincar\'{e} metric given by \eqref{eq:P metc}. Since $k=K=1$ in Fact \ref{fac:Yau-Schwarz}, we obtain
\begin{align*}
\mathrm{id}^* \, (ds_c^2) \leq ds^2, 
\quad \mathrm{i.e.,} \quad \dfrac{2}{1-|z|^2} \leq \lambda(z) \quad \mathrm{on} \; \D.
\end{align*}

\end{proof}

We now verify that the compact property $P_X \; (|X| \geq m + 3)$ satisfies $m$-curvature estimate (\cite[Theorem 2.1]{YK2013}).

\begin{theorem} \label{thm:Fujimoto}
Let $m \in \Z_{> 0}$, $X \subset \RC$ and $P_X$ be the property in Theorem \ref{thm:Picard Montel}.
If the set $X$ contains $m + 3$ or more elements, then $P_X$ is a compact property that satisfies $m$-curvature estimate. 
In particular, for any $m \in \Z_{>0}$ and for any $X \subset \RC$ with $|X| \geq m + 3$, there exists a positive constant $C = C(m, X) > 0$ such that for any Weierstrass $m$-triple $(\Sigma, \, f \, dz, \, g)$ whose $g$ omits all elements of $X$, the following inequality holds:
\begin{align*}
|K_{ds^2}(p)|^{\frac{1}{2}} \leq \dfrac{C}{d(p)} \quad(^\forall p \in \Sigma).
\end{align*}
\end{theorem}

\begin{proof}
If $|X| = m + 3 \, (\geq 3)$, that is, $X=\{ \alpha_1, \, \alpha_2, \, \cdots, \, \alpha_{m+3} \}$, we prove that $P_X$ is a compact property satisfying $m$-curvatue estimate by contradiction. 
Suppose that $P_X$ is a compact property that does not satisfy $m$-curvature estimate. 
By Theorem \ref{thm:m-curvature est}, there exist a Weierstrass $m$-triple $(\D, \, f \, dz, \, g)$ such that (A), (B) and (C) in Theorem \ref{thm:m-curvature est} hold. 
Remark that $g$ omits all elements of $X$. 
Here, we may assume that $\alpha_{m+3}=\infty$ by composing with a suitable M\"{o}bius transformation if necessary. 

We choose a positive number $\eta$ with 
\[ 
0 < \eta < \dfrac{1}{m+3} \left( \leq \dfrac{(m+3)-2}{m+3} \right). 
\]
By Fact \ref{fac:Fujimoto lemma}, there exists $C >0$ such that  
\begin{align} \label{eq:first Fujimoto}
\dfrac{|g'|}{(1+|g|^2)\prod_{j=1}^{m+3}\chi(g,\alpha_j)^{1-\eta}} \leq \dfrac{C}{1-|z|^2} \quad \mathrm{on}\;\D.
\end{align}
From $\chi(g,\alpha_j) \leq |g-\alpha_j|/\sqrt{1+|g|^2} \; (1 \leq j \leq m + 2)$ and $\chi(g,\alpha_{m+3})=1/\sqrt{1+|g|^2}$,
\begin{align*}
(1+|g|^2)\,\prod_{j=1}^{m+3}\chi(g,\alpha_j)^{1-\eta} 
\leq \left(\prod_{j=1}^{m+2}|g-\alpha_j|\right)^{1-\eta}\,(1+|g|^2)^{1-\frac{1}{2}(m+3)(1-\eta)}.
\end{align*}
By this inequality and \eqref{eq:first Fujimoto}, we obtain
\begin{align} \label{eq:second Fujimoto}
\dfrac{(1+|g|^2)^{\frac{1}{2}(m+3)(1-\eta)-1}\,|g'|}{\left(\prod_{j=1}^{m+2}|g-\alpha_j|\right)^{1-\eta}} 
\leq \dfrac{|g'|}{(1+|g|^2)\,\prod_{j=1}^{m+3}\chi(g,\alpha_j)^{1-\eta}}
\leq \dfrac{C}{1-|z|^2} \quad \mathrm{on}\;\D. 
\end{align}
Note that $f$ has no zeros on $\D$ by \eqref{eq:regularity}. 
By applying Lemma \ref{lem:Yau-Schwarz} to the metric \eqref{eq:m-metric}, we have
\begin{align} \label{eq:Yau}
\dfrac{2}{1-|z|^2} \leq |f| \, (1+|g|^2)^{\frac{m}{2}}, 
\quad \mathrm{i.e.,} \quad \dfrac{1}{|f|} \, \dfrac{1}{1-|z|^2} < (1+|g|^2)^{\frac{m}{2}} \quad \mathrm{on} \; \D.
\end{align}
Since 
\[
\left\{\dfrac{1}{2}(m+3)(1-\eta)-1\right\}\,\dfrac{2}{m} >0, 
\]
\eqref{eq:Yau} is transformed as follows:
\begin{align*}
\left( \dfrac{1}{|f|}\,\dfrac{1}{1-|z|^2} \right)^{ \left\{ \frac{1}{2}(m+3)(1-\eta)-1 \right\}\frac{2}{m} } \leq (1+|g|^2)^{\frac{1}{2}(m+3)(1-\eta)-1} \quad \mathrm{on}\;\D.
\end{align*}
By this inequality and \eqref{eq:second Fujimoto}, we have
\begin{align*}
\left( \dfrac{1}{|f|}\,\dfrac{1}{1-|z|^2} \right)^{ \left\{ \frac{1}{2}(m+3)(1-\eta)-1 \right\}\frac{2}{m} } \, \dfrac{|g'|}{\left(\prod_{j=1}^{m+2}|g-\alpha_j|\right)^{1-\eta}} \leq \dfrac{C}{1-|z|^2} \quad \mathrm{on}\;\D.
\end{align*}
Set $D':=\{ z \in \D : g'(z) \ne 0 \}(\subset \D)$. 
We thereby obtain
\begin{align*}
\dfrac{1}{C}\,\dfrac{1}{(1-|z|^2)^{ \left\{ \frac{1}{2}(m+3)(1-\eta)-1 \right\}\frac{2}{m}-1 }} 
\leq \dfrac{|f|^{ \left\{ \frac{1}{2}(m+3)(1-\eta)-1 \right\}\frac{2}{m} }\,\left(\prod_{j=1}^{m+2}|g-\alpha_j|\right)^{1-\eta}}{|g'|} \quad \mathrm{on}\;D'.
\end{align*}
Since 
\[
r:=\dfrac{1}{\left\{ \frac{1}{2} (m+3)(1-\eta)-1 \right\} \frac{2}{m}-1} \;>0, 
\] 
we have
\begin{align} \label{eq:final Fujimoto}
\dfrac{C^{-r}}{1-|z|^2} \leq \left( \dfrac{|f|^{ \left\{ \frac{1}{2}(m+3)(1-\eta)-1 \right\}\frac{2}{m} }\,\left(\prod_{j=1}^{m+2}|g-\alpha_j|\right)^{1-\eta}}{|g'|} \right)^r \quad \mathrm{on}\;D'.
\end{align}
Let $\lambda$ be the right hand side of $\eqref{eq:final Fujimoto}$ and we consider the conformal metric $ds_0^2 := \lambda^2 \, |dz|^2$ on $D'$. 
By applying $\Delta \log|h| = 0$ for any holomorphic function $h$ without zeros, we obtain $K_{ds_0^2}=-\Delta \log \lambda/\lambda^2 \equiv 0$ on $D'$, that is, $ds_0^2$ is a flat metric on $D'$. 
Furthermore, $ds_0^2$ is a complete metric on $D'$. This is because the left hand side of \eqref{eq:final Fujimoto} is the Poincar\'{e} metric and each puncture of $D'$ is a point where $\lambda$ diverges. Therefore the universal covering of $(D',ds_0^2)$ is $\C$. This is a contradiction. 
\end{proof}

As a corollary of Theorem \ref{thm:Fujimoto}, we give the following Picard-type result (\cite[Corollary 2.2, Proposition 2.4]{YK2013}).

\begin{corollary}\label{cor:Fujimoto}
Let $(\Sigma, \, f \, dz, \, g)$ be a Weierstrass $m$-triple.
If the metric given by \eqref{eq:m-metric} is complete on $\Sigma$ and $g$ is nonconstant on $\Sigma$, then $g$ can omit at most $m + 2$ distinct values. 
Furthermore, this result is optimal.
\end{corollary}

\begin{proof}
We give a proof of the first claim by contradiction. 
Suppose that there exists a Weierstrass $m$-triple $(\Sigma, \, f \, dz, \, g)$ such that the metric $ds^2$ given by \eqref{eq:m-metric} is complete on $\Sigma$ and $g$ is a nonconstant meromorphic function that omits $m + 3$ or more values. 
By Theorem \ref{thm:Fujimoto}, there exists a positive constant $C$ such that $|K_{ds^2}(p)|^{\frac{1}{2}} \leq C/d(p) \;(^\forall p \in \Sigma)$. 
Since $ds^2$ is complete on $\Sigma$, we obtain $d(p)=\infty$ for any $p \in \Sigma$. 
Hence we have $K_{ds^2} \equiv 0$ on $\Sigma$, that is, $g$ is constant on $\Sigma$. 
This is a contradiction.  

We next give a proof of the second claim. 
Let $\Sigma$ be $\C\setminus \{ \alpha_1, \, \alpha_2, \, \cdots, \, \alpha_{m+1} \}$ or the universal covering of $\C\setminus \{ \alpha_1, \, \alpha_2, \, \cdots, \, \alpha_{m+1} \}$. 
Set 
\begin{align*}
f \, dz := \dfrac{dz}{\prod_{j=1}^{m+1} (z-\alpha_j) },
\quad g := z.
\end{align*}
Then $g$ omits $m+2$ distinct values $\alpha_1, \, \alpha_2, \, \cdots, \, \alpha_{m+1}, \, \infty$. 
Furthermore, the metric $ds^2 = (1+|g|^2)^m \, |f|^2 \, |dz|^2$ is complete on $\Sigma$. 
This is because we have
\begin{align*}
\int_\gamma ds = \int_\gamma \, \dfrac{(1+|z|^2)^{\frac{m}{2}}}{\prod_{j=1}^{m+1} |z-\alpha_j|} \, |dz| = +\infty
\end{align*}
for any curve $\gamma$ tending to one of the points $\alpha_1, \, \alpha_2, \, \cdots, \, \alpha_{m+1}, \, \infty$.
\end{proof}

\subsection{Proof of Theorem \ref{thm:m-curvature est}} \label{sbseq:3.2} 

We first recall the following lemma.

\begin{fact} \cite[Lemma 2.1]{OR} \label{fac:Oss-Ru's lemma}
Let $0 < r < 1$ and $R$ be the hyperbolic radius of $D(0;r)$ in the unit disk, that is, $R := \log \, \{ (1+r)/(1-r) \}$. 
Let $ds^2 = (\lambda(z))^2 \, |dz|^2$ be any conformal metric on $D(0;r)$ with the property that $(the \; geodesic \; distance \; from \; z=0 \; to \; |z|=r) \geq R$ and $-1 \leq K_{ds^2} \leq 0$, where $K_{ds^2}$ is the Gaussian curvature with respect to the metric $ds^2$. Then, for all $|z| \leq r$, the following inequality holds:
\begin{align*}
& \mathrel{\phantom{=}} (the \; distance \; between \; 0 \; and \; z \; in \; the \; metric \; ds^2) \\
& \geq (the \; hyperbolic \; distance \; between \; 0 \; and \; z) = \log\,\dfrac{1+|z|}{1-|z|}.
\end{align*}
\end{fact}

The first author formulates the following lemma with reference to \cite[Lemma 8.5 and Theorem 11.1]{Oss1986}.  

\begin{lemma} \label{lem:point curvature est.}
For any $L>0$ and for each $m \in \Z_{>0}$, there exists $C=C(L,m)>0$ such that for any Weierstrass $m$-triple $(D(0;R), \, f \, dz, \, g)$ satisfying $|g| < L$ on $D(0;R)$, we have
\begin{align*}
|K_{ds^2}(0)|^{\frac{1}{2}} \leq \dfrac{C}{d(0)}.
\end{align*}
Here, $K_{ds^2}(0)$ is the Gaussian curvature with respect to the metric $ds^2$ given by \eqref{eq:m-metric} at $0 \in D(0;R)$ and $d(0)$ is the geodesic distance from $0$ to the boundary $\partial D(0;R)$. 
\end{lemma}

\begin{proof}
For any $L > 0$ and for any $m \in \Z_{>0}$, set $C := \sqrt{2 \, m} \, L \, (1+L^2)^{\frac{m}{2}} \, (>0)$. 
It is sufficient to show that $|K_{ds^2}(0)|^{\frac{1}{2}} \, d(0) \leq C$ for any Weierstrass $m$-triple $(D(0;R), \, f \, dz, \, g)$ satisfying $|g| < L$ on $D(0;R)$. 
So we take an arbitrary Weierstrass $m$-triple $(D(0;R), \, f \, dz, \, g)$ satisfying $|g| < L$ on $D(0;R)$. 
For any divergent path $\gamma$ starting from the origin in $D(0;R)$, we have
\begin{align} \label{eq:d(0)}
d(0) 
= \inf_{\gamma} \int_\gamma ds 
\leq \int_\gamma (1+|g|^2)^{\frac{m}{2}} \, |f| \, |dz| 
\leq (1+L^2)^{\frac{m}{2}} \int_\gamma |f(z)| \, |dz|. 
\end{align}

We next construct the divergent path. By the Cauchy integral theorem, the line integral of $f$ connecting $0$ and $z \in D(0;R)$ does not depend on paths. 
Then we define $F:D(0;R)  \longrightarrow \C \; ; \; w = F(z) := \int_0^z f(\zeta) \, d\zeta$. The function $F$ is a holomorphic function satisfying $F(0)=0$ and $F'=f$ on $D(0;R)$. 
By $|g|<L$ on $D(0;R)$ and \eqref{eq:regularity}, $f$ has no zeros on $D(0;R)$. 
Since $F'=f \ne 0$ on $D(0;R)$, $F$ is locally univalent on $D(0;R)$, that is, 
\begin{align*}
(LU)\quad ^\forall z \in D(0;R), \, ^\exists U\,:\,\mathrm{open \; neighborhood \; of \;}z \;\, \mathrm{s.t.} \;\, F|_U:U \longrightarrow F(U)\,: \mathrm{conformal}.
\end{align*}
By applying the Liouville theorem to the inverse map of $F$, there exists $\rho>0$ such that $V_0:=\{ w \in \C\, : |w|< \rho \}$ is the largest disk satisfying the condition $(LU)$. 
Let $U_0:=F^{-1}(V_0)$, $G:V_0 \longrightarrow U_0 \; ; \; z = G(w)$ be the inverse function of $F$ and $w_0 \in \C$ be the point such that $|w_0|=\rho$ and $G$ cannot be extended at $w_0$. 
Moreover, we set the curves
\begin{align*}
l : [0,1) \longrightarrow V_0 \; ; \;\gamma(t) = t \, w_0, \quad 
C : [0,1) \longrightarrow U_0 \; ; \; G \circ \gamma(t) = G(t \, w_0).
\end{align*}

Then we can show that $C$ is a divergent path in $D(0;R)$. 
In fact, suppose that $C$ is not a divergent path in $D(0;R)$, that is, there exists a compact subset $K_0$ in $D(0;R)$ such that $(G \circ \gamma)([0,1)) \subset K_0$. 
We choose $\{ t_n \}_{n=1}^\infty \subset [0,1)$ with $t_n \rightarrow 1$ (as $n \rightarrow \infty$) and we consider $\{(G \circ \gamma)(t_n)\}_{n=1}^\infty (\subset (G \circ \gamma)([0,1)) \subset K_0)$. 
Since $K_0$ is compact, by taking a subsequence if necessary, there exists $z_0 \in D(0;R)$ such that $(G \circ \gamma)(t_n) \rightarrow z_0$. 
Since $F$ is continuous at $z_0 \in D(0;R)$, we have 
\begin{align*}
\lim_{n \rightarrow \infty} F((G \circ \gamma)(t_n))=F(\lim_{n \rightarrow \infty}(G \circ \gamma)(t_n))=F(z_0).
\end{align*}
On the other hand, by $F \circ G = \mathrm{id}$, we obtain 
\begin{align*}
\lim_{n \rightarrow \infty} F((G \circ \gamma)(t_n)) 
= \lim_{n \rightarrow \infty} \gamma(t_n) = w_0.
\end{align*}
We thereby have $F(z_0) = w_0$. 
Moreover, since $F:\D \longrightarrow \C$ is an open map and $\overline{U_0}$ is a closed set which contains $\{ (G \circ \gamma)(t_n) \}_{n=1}^\infty$, we see that $\overline{U_0} \cap F^{-1}(w_0)=\{z_0\}$. 
This contradicts the definition of $w_0$. 

From $G'=1/F' = 1/f$ and $|\gamma'(t)|=|w_0|=\rho$, we have  
\begin{align*}
\int_C |f(z)| \, |dz|
=\int_0^1 |f(G \circ \gamma(t))| \, \left| \dfrac{dG}{dw}(\gamma(t)) \right| \, \left| \dfrac{d\gamma}{dt}(t) \right| \, dt 
= \rho.
\end{align*}
By \eqref{eq:d(0)}, we have
\begin{align} \label{eq:second d(0)}
d(0) \leq (1+L^2)^{\frac{m}{2}} \, \rho.
\end{align}
Applying the Schwarz lemma to $G:D(0;\rho) \longrightarrow D(0;R)$, 
\begin{align} \label{eq:G'(0)}
|G'(0)| \leq \dfrac{R}{\rho}.
\end{align}
By using $f(0)=1/G'(0)$, \eqref{eq:G'(0)} and \eqref{eq:second d(0)} in this order, we obtain
\begin{align} \label{eq:final d(0)}
(1+L^2)^{\frac{m}{2}} \, |f(0)| 
= (1+L^2)^{\frac{m}{2}} \, \dfrac{1}{|G'(0)|}
\geq (1+L^2)^{\frac{m}{2}} \, \dfrac{\rho}{R}
\geq \dfrac{d(0)}{R}.
\end{align}
Applying the Schwarz--Pick lemma to $g:D(0;R) \longrightarrow D(0;L)$, we also have
\begin{align} \label{eq:g'(0)}
|g'(0)| \leq \dfrac{L}{R}, \quad \mathrm{i.e.,} \quad |g'(0)| \, R \leq L.
\end{align}
By using \eqref{eq:Gauss curvature of m-Riemann surf.}, \eqref{eq:final d(0)}, \eqref{eq:g'(0)} and $1 \leq (1+|g(0)|^2)^{\frac{m}{2}+1}$ in this order, we therefore obtain
\begin{align*}
|K_{ds^2}(0)|^{\frac{1}{2}} \, d(0) 
& = \dfrac{\sqrt{2 \, m} \, |g'(0)|}{(1+|g(0)|^2)^{\frac{m}{2}+1} \, |f(0)|} \, d(0) \\
& \leq \dfrac{\sqrt{2 \, m} \, |g'(0)|}{(1+|g(0)|^2)^{\frac{m}{2}+1}} \, R \, (1+L^2)^\frac{m}{2} \\
& \leq \dfrac{\sqrt{2 \, m} \, L \, (1+L^2)^\frac{m}{2}}{(1+|g(0)|^2)^{\frac{m}{2}+1}} \\
& \leq \sqrt{2 \, m} \, L \, (1+L^2)^\frac{m}{2} \\
& = C.
\end{align*}
\end{proof}

By generalizing \cite[Lemma 6]{RA} to Weierstrass $m$-triples, we obtain the following lemma. 

\begin{lemma} \label{lem:conv of m-Riemann surf}
Let $\{ (\Sigma,f_ndz,g_n) \}_{n=1}^\infty$ be a sequence of Weierstrass $m$-triples 
and $\{ K_{ds_n^2} \}_{n=1}^{\infty}$ be the sequence of the Gaussian curvatures of the metric $ds_n^2 = (1 + |g_n|^2)^m \, |f_n|^2 \, |dz|^2$. 
Suppose that there exists a nonconstant meromorphic function $g \in \M(\Sigma)$ such that $g_n \stackrel{\mathrm{sph.\;loc.}}{\rightrightarrows} g \; \mathrm{on} \;\Sigma$ and $\{ K_{ds_n^2} \}_{n=1}^\infty$ is uniformly bounded on $\Sigma$. 
Then one of the following is true: 
\begin{enumerate}
\item[(\hspace{.18em}i\hspace{.18em})] There exists a subsequence $\left\{ K_{ds_{n_k}^2} \right\}_{k=1}^\infty$ of $\{ K_{ds_n^2} \}_{n=1}^\infty$ such that $K_{ds_{n_k}^2} \stackrel{\mathrm{loc.}}{\rightrightarrows} 0 \; \mathrm{on} \;\Sigma$.  
\item[(\hspace{.08em}ii\hspace{.08em})] There exists a subsequence $\{ ds_{n_k}^2 \}_{k=1}^\infty$ of $\{ ds^2_n \}_{n=1}^\infty$ and a holomorphic $1$-form $f \, dz$ on $\Sigma$ such that $ds^2 := (1+|g|^2)^m \, |f|^2 \, |dz|^2$ is a conformal metric on $\Sigma$ and $ds_{n_k}^2 \stackrel{\mathrm{loc.}}{\rightrightarrows} ds^2 \; \mathrm{on} \;\Sigma$.
\end{enumerate}
\end{lemma}

\begin{proof}
Since $K_{ds_n^2}$ is nonpositive on $\Sigma$ and uniformly bounded on $\Sigma$, by rescaling if necessary, we may assume that $-1 \leq K_{ds_n^2}(p) \leq 0$ for each $n \in \Z_{>0}$ and for any $p \in \Sigma$. 
Let $(\Omega_p,z)$ be a complex local coordinate centered at $p$. 
Take an arbitrary point $p$ where $g$ is unbranched and $g(p) \ne \infty$. 
By Fact \ref{fact:conv of mero}, there exist a relatively compact neighborhood $\Omega_p$ of $p$ and $N \in \Z_{>0}$ such that $g_n \; (n \geq N)$ and $g$ are bounded holomorphic functions without ramification points on $\Omega_p$ and $g_n \rightrightarrows g \; \mathrm{on} \; \Omega_p$.
Remark that $2 \, m \, |g'|^2/(1+|g|^2)^{m+2} > 0$ on $\Omega_p$, that is, there exists $\varepsilon > 0$ such that 
\begin{align*}
\dfrac{2 \, m \, |g'|^2}{(1+|g|^2)^{m+2}} \geq 2 \, \varepsilon^2.
\end{align*}
We also note that $2 \, m \, |g_n'|^2/(1+|g_n|^2)^{m+2} \rightrightarrows 2 \, m \, |g'|^2/(1+|g|^2)^{m+2} \; \mathrm{on} \; \Omega_p$. 
Hence there exists $N' \in \Z_{\geq N}$ such that for any $n \geq N'$,
\begin{align*}
\dfrac{2 \, m \, |g_n'|^2}{(1+|g_n|^2)^{m+2}} \geq \varepsilon^2 \quad \mathrm{on} \; \Omega_p.
\end{align*}
Furthermore, $f_n \; (n \geq N')$ has no zeros on $\Omega_p$ by \eqref{eq:regularity}. 
So we see that for any $n \geq N'$,
\begin{align*} 
\dfrac{\varepsilon^2}{|f_n|^2} 
\leq \dfrac{2 \, m \, |g_n'|^2}{(1+|g_n|^2)^{m+2} \, |f_n|^2} 
= |K_{ds_n^2}| 
\leq 1, \quad
\mathrm{i.e.,} \quad \dfrac{1}{|f_n|} \leq \dfrac{1}{\varepsilon} \quad \mathrm{on} \; \Omega_p.
\end{align*}
Applying the Montel theorem and the Hurwitz theorem, the sequence $\{ f_n \}_{n=1}^\infty$ of holomorphic functions is normal on $\Omega_p$. 
More precisely, $\{ f_n \}_{n=1}^\infty$ contains a subsequence that converges locally uniformly to either a holomorphic function without zeros on $\Omega_p$ or $\infty$. 
So we define 
\begin{align*}
\Sigma_1 
:= \{ p \in \Sigma : p\;\mathrm{is \; neither \; a \; pole \; of} \; g  \; \mathrm{ nor \; a \; ramification \; point \; of}\; g \} \subset \Sigma.
\end{align*}
Then the sequence $\{ f_n \, dz \}_{n=1}^\infty$ of holomorphic $1$-forms on $\Sigma$ is normal on $\Sigma_1$ because of Remark \ref{rem:normality of mero}. Here, we remark that 
\begin{align*}
\Sigma \setminus \Sigma_1 = \{ p \in \Sigma \, : \, g(p) = \infty \} \cup \{ p \in \Sigma \, : \, p \; \mathrm{is \; a \; ramification \; point \; of} \; g \}
\end{align*}
is a discrete set. 
By taking a subsequence if necessary, one of the following is true: 

(case \hspace{.18em}i\hspace{.18em}) $f_n \, dz \stackrel{\mathrm{loc.}}{\rightrightarrows} \infty \; \mathrm{on} \; \Sigma_1$.

(case \hspace{.08em}ii\hspace{.08em}) There exists a holomorphic $1$-form $f \, dz$ without zeros on $\Sigma_1$ such that $f_n \, dz \stackrel{\mathrm{loc.}}{\rightrightarrows} f \, dz \; \mathrm{on} \; \Sigma_1$.

We first consider (case $\rm\hspace{.18em}i\hspace{.18em}$). 
Let $p$ be a point of $\Sigma \setminus \Sigma_1$. 
Then we have the following two possibilities and we can prove $K_{ds_n^2} \stackrel{\mathrm{loc.}}{\rightrightarrows} 0 \; \mathrm{on} \;\Sigma$ in both cases. 

(case \hspace{.18em}i\hspace{.18em}-a) $g(p) \ne \infty$, that is, $p$ is a ramification point of $g$. 

(case \hspace{.18em}i\hspace{.18em}-b) $g(p) = \infty$.

In (case \hspace{.18em}i\hspace{.18em}-a), by the discreteness of $\Sigma \setminus \Sigma_1$ and Fact \ref{fact:conv of mero} $(\rm\hspace{.18em}i\hspace{.18em})$, there exist a relatively compact neighborhood $\Omega_p$ of $p$ and $n_0 \in\Z_{>0}$ such that $\overline{\Omega_p} \cap (\Sigma \setminus \Sigma_1) = \{ p \}$, $g_n \; (n \geq n_0)$ and $g$ are holomorphic on $\Omega_p$. 
By \eqref{eq:regularity}, $f_n \; (n \geq n_0)$ are holomorphic functions without zeros on $\Omega_p$.
Take a neighborhood $N_p$ of $p$ such that $\overline{N_p} \subset \Omega_p$. 
Since $\partial N_p \subset \Sigma_1$ and $\partial N_p$ is compact, 
\begin{align*}
f_n \rightrightarrows \infty \; \mathrm{on} \;\partial N_p,
\quad \mathrm{i.e.,} \quad \lim_{n \rightarrow \infty} \sup_{\partial N_p} \left\{ \dfrac{1}{|f_n|} \right\}=0.
\end{align*}
Applying the maximum modulus principle to each holomorphic function $1/f_n \; (n \geq n_0)$, we have
\begin{align*}
\sup_{\overline{N_p}} \left\{ \dfrac{1}{|f_n|} \right\}
= \max_{\overline{N_p}} \left\{ \dfrac{1}{|f_n|} \right\}
= \max_{\partial N_p} \left\{ \dfrac{1}{|f_n|} \right\}
= \sup_{\partial N_p} \left\{ \dfrac{1}{|f_n|} \right\}.
\end{align*}
We thereby obtain $\lim_{n \rightarrow \infty}\sup_{\overline{N_p}} \left\{ 1/|f_n| \right\}=0$, that is, $f_n \rightrightarrows \infty \; \mathrm{on} \;\overline{N_p}$.

In (case \rm\hspace{.18em}i\hspace{.18em}-b), by the discreteness of $\Sigma \setminus \Sigma_1$ and Fact \ref{fact:conv of mero} $(\rm\hspace{.08em}ii\hspace{.08em})$, there exist a relatively compact neighborhood $\Omega_p$ of $p$ and $n_0 \in\Z_{>0}$ such that $\overline{\Omega_p} \cap (\Sigma \setminus \Sigma_1) = \{ p \}$, $1/g_n \; (n \geq n_0)$ and $1/g$ are bounded holomorphic functions on $\Omega_p$. 
In particular, we remark that $g_n \; (n \geq n_0)$ has no zeros on $\Omega_p$. 
We also recall \eqref{eq:regularity}. 
Hence, $g_n^m \, f_n \; (n \geq n_0) $ are holomorphic functions without zeros on $\Omega_p$. 
Take a neighborhood $N_p$ of $p$ such that $\overline{N_p} \subset \Omega_p$. 
Since $\partial N_p \subset \Sigma_1$ and $\partial N_p$ is compact, $g_n^m \, f_n \rightrightarrows \infty \quad \mathrm{on} \; \partial N_p$.
Applying the maximum modulus principle to each holomorphic function $1/(g_n^m \, f_n)$ ($n \geq n_0$) as in (case \hspace{.18em}i\hspace{.18em}-a), we have $g_n^m \, f_n \rightrightarrows \infty$ on $\overline{N_p}$. 

We therefore see that for any $p \in \Sigma$, there exists a neighborhood $\Omega_p$ such that $(1+|g_n|^2)^m \, |f_n|^2 \rightrightarrows \infty$ on $\Omega_p$. 
This is because $(1+|g_n|^2)^m \, |f_n|^2 \geq |f_n|^2$ and $(1+|g_n|^2)^m \, |f_n|^2 \geq |g_n|^{2m} \, |f_n|^2$ hold. 
Furthermore, by taking $\Omega_p$ sufficiently small if necessary, $|\nabla g_n|_e \rightrightarrows |\nabla g|_e \; \mathrm{on} \; \Omega_p$. 
Thus we obtain 
\begin{align*}
|K_{ds_n^2}|
= \dfrac{m \, |\nabla g_n|_e^2}{4 \, (1+|g_n|^2)^m \, |f_n|^2} \rightrightarrows 0 \quad \mathrm{on} \; \Omega_p.
\end{align*}
Since $p$ is an arbitrary point of $\Sigma$, we obtain $K_{ds_n^2} \stackrel{\mathrm{loc.}}{\rightrightarrows} 0$ on $\Sigma$.  

We next consider (case $\rm\hspace{.08em}ii\hspace{.08em}$). 
In the following, we will show that $f \, dz$ is a holomorphic $1$-form on $\Sigma$, $ds^2 := (1+|g|^2)^m \, |f|^2 \, |dz|^2$ is a conformal metric on $\Sigma$ and $ds_n^2 \stackrel{\mathrm{loc.}}{\rightrightarrows} ds^2 \; \mathrm{on} \;\Sigma$. 
Let $p$ be a point of $\Sigma \setminus \Sigma_1$. 
Since $\Sigma \setminus \Sigma_1$ is discrete, we may take a relatively compact neighborhood $\Omega_p$ of $p$ such that $\overline{\Omega_p} \cap (\Sigma \setminus \Sigma_1) = \{ p \}$, that is, $\overline{\Omega_p} \setminus \{ p \} \subset \Sigma_1$. 
Remark that $f$ has no zeros on $\overline{\Omega_p} \setminus \{ p \}$. 
Moreover, $f_n \rightrightarrows f \; \mathrm{on} \; \partial \Omega_p$ because $\partial \Omega_p \subset \Sigma_1$ and $\partial \Omega_p$ is compact. 
Then $\{ f_n \}_{n=1}^\infty$ is uniformly bounded on $\partial \Omega_p$, that is, there exists $M > 0$ such that for each $n \in \Z_{>0}$, $|f_n|<M$ on $\partial \Omega_p$. 
Applying the maximum modulus principle to each $f_n$, for any $n \in \Z_{>0}$, $|f_n|<M$ on $\overline{\Omega_p}$. 
By the Montel theorem, $\{ f_n \}_{n=1}^\infty$ is normal on $\Omega_p$. 
Since $p$ is an arbitrary point of $\Sigma \setminus \Sigma_1$, $\{ f_n \, dz \}_{n=1}^\infty$ is normal on $\Sigma$. 
More preciously, $f_n \, dz \stackrel{\mathrm{loc.}}{\rightrightarrows} f \, dz \; \mathrm{on} \;\Sigma$ and $f \, dz$ is a holomorphic $1$-form on $\Sigma$ and has no zeros on $\Sigma_1$. 

We show the following claims (A) and $(B)$ by using the Hurwitz theorem and Fact \ref{fact:conv of mero}: 
\begin{enumerate}
\item[(A)] The function $f$ vanishes only at the poles of $g$. 
\item[(B)] If $p$ is a pole of $g$ of order $k$, then $p$ is a zero of $f$ of order $k m$.
\end{enumerate}
We first prove (A). 
We have already seen that if $f \, dz$ has zeros, these belong to $\Sigma \setminus \Sigma_1$. 
So it is sufficient to show that zeros of $f \, dz$ do not belong to the set of ramification points (except for poles) of $g$. 
Suppose that there exists a zero $p$ of $f \, dz$ such that $p$ is a ramification point of $g$ and not a pole of $g$. 
By the discreteness of $\Sigma \setminus \Sigma_1$ and Fact \ref{fact:conv of mero} $(\rm\hspace{.18em}i\hspace{.18em})$, there exist a neighborhood $\Omega_p$ of $p$ and $n_0 \in \Z_{>0}$ such that $\overline{\Omega_p} \cap (\Sigma \setminus \Sigma_1) = \{ p \}$ and $g_n \; (n \geq n_0)$ and $g$ are holomorphic on $\Omega_p$.
On the other hand, by the Hurwitz theorem, there exists $\widetilde{n_0} \in \Z_{\geq n_0}$ such that for each $n \geq \widetilde{n_0}$ 
\begin{align*}
(\mathrm{the \; order \; of \; zeros \; of} \; f_n \; \mathrm{on} \; \Omega_p )
= (\mathrm{the \; order \; of \; zeros \; of} \; f \; \mathrm{on} \; \Omega_p )
= \mathrm{ord}_{f}(p) 
\geq 1.
\end{align*}
By \eqref{eq:regularity}, $g_n \; (n \geq \widetilde{n_0})$ has poles on $\Omega_p$.
This is a contradiction. 
We next prove (B). 
Let $p \in \Sigma$ be a pole of $g$ of order $k$. 
By the discreteness of $\Sigma \setminus \Sigma_1$ and Fact \ref{fact:conv of mero} $(\rm\hspace{.08em}ii\hspace{.08em})$, there exist a relatively compact neighborhood $\Omega_p$ of $p$ and $n_0 \in \Z_{>0}$ such that $\overline{\Omega_p} \cap (\Sigma \setminus \Sigma_1) = \{ p \}$ and $1/g_n$ $(n \geq n_0)$, $1/g$ are holomorphic on $\Omega_p$ and $1/g_n \rightrightarrows 1/g \; \mathrm{on} \; \Omega_p$.
By applying the Hurwitz theorem, there exists $\widetilde{n_0} \in \Z_{\geq n_0}$ such that for any $n \geq \widetilde{n_0}$, 
\begin{align*}
\left(\mathrm{the \; order \; of \; zeros \; of} \; \dfrac{1}{g_n} \; \mathrm{on} \; \Omega_p \right) 
= \left(\mathrm{the \; order \; of \; zeros \; of} \; \dfrac{1}{g} \; \mathrm{on} \; \Omega_p \right)=\mathrm{ord}_{\frac{1}{g}}(p) 
= k.
\end{align*}
From \eqref{eq:regularity}, the order of zeros of $f_n$ ($n \geq \widetilde{n_0}$) on $\Omega_p$ is $k m$. 
Then we find that $p$ is a zero of $f$ by contradiction. 
In fact, suppose that $p$ is not a zero of $f$.
Since $f_n \rightrightarrows f \; \mathrm{on} \; \overline{\Omega_p}$, $f_n \stackrel{\mathrm{sph.}}{\rightrightarrows} f \; \mathrm{on} \; \overline{\Omega_p}$, i.e., $1/f_n \stackrel{\mathrm{sph.}}{\rightrightarrows} 1/f \; \mathrm{on} \; \overline{\Omega_p}$. 
We also remark that $1/f$ is bounded on $\overline{\Omega_p}$, that is, there exists $M > 0$ such that $1/|f| < M$ on $\overline{\Omega_p}$.
By applying Fact \ref{fact:conv}, $1/f_n \rightrightarrows 1/f \; \mathrm{on} \; \overline{\Omega_p}$, that is, there exists $n \in \Z_{\geq \widetilde{n_0}}$ such that 
\begin{align*}
  \dfrac{1}{|f_n|} \leq \dfrac{1}{|f|} + 1 < M +1 \quad \mathrm{on} \; \overline{\Omega_p}.
\end{align*}
This contradicts the claim that $f_n \; (n \geq \widetilde{n_0})$ has zeros on $\Omega_p$. 
Hence $p$ is a zero of $f$.
Furthermore, by the Hurwitz theorem, there exists $n \in \Z_{\geq \widetilde{n_0}}$ such that
\begin{align*}
\mathrm{ord}_f(p) 
= (\mathrm{the \; order \; of \; zeros \; of} \; f \; \mathrm{on} \; \Omega_p ) 
= (\mathrm{the \; order \; of \; zeros \; of} \; f_n \; \mathrm{on} \; \Omega_p )
= km.
\end{align*}

By \eqref{eq:regularity}, (A) and (B), $ds^2 = (1+|g|^2)^m \, |f|^2 \, |dz|^2$ is a conformal metric on $\Sigma$. 
We also find that for each $0 \leq l \leq m$, $g_n^l \, f_n \, dz \stackrel{\mathrm{loc.}}{\rightrightarrows} g^l \, f \, dz \; \mathrm{on} \; \Sigma$. 
Hence, for any $p \in \Sigma$, there exists a neighborhood $\Omega_p$ of $p$ such that
\begin{align*}
  g_n^l \, f_n \rightrightarrows g^l \, f \quad \mathrm{on} \;\, \Omega_p, \quad \mathrm{i.e.,} \quad |g_n|^{2l} \, |f_n|^2 \rightrightarrows |g|^{2l} \, |f|^2 \quad \mathrm{on} \;\, \Omega_p.
\end{align*}
We therefore obtain $ds_n^2 \stackrel{\mathrm{loc.}}{\rightrightarrows} ds^2 \; \mathrm{on} \; \Sigma$.
\end{proof}

Now, we give a proof of Theorem \ref{thm:m-curvature est}.
Remark that Lemma \ref{lem:point curvature est.} in this section gives only the curvature estimate at the origin.

\begin{proof}[\textit{Proof of Theorem \ref{thm:m-curvature est}}]
Let $P$ be a compact property that does not satisfy $m$-curvature estimate. 
Then there exists a sequence $\{(\Sigma_n, \, f_n \, dz,  \,g_n)\}_{n=1}^\infty$ of Weierstrass $m$-triples and points $p_n \in \Sigma_n$ such that $g_n \in \P(\Sigma_n)$ and $|K_{ds_n^2}(p_n)| \, d_n(p_n)^2 \rightarrow \infty$. 

We claim that the sequence $\{(\Sigma_n, \, f_n \, dz,  \,g_n)\}_{n=1}^\infty$ of Weierstrass $m$-triples can be chosen so that
\begin{align*} 
K_{ds_n^2}(p_n)=-\dfrac{1}{4},
\;\, -1 \leq K_{ds_n^2} \leq 0 \;\mathrm{on}\; \Sigma_n \;\;(^\forall n \in \Z_{>0})
\;\, \mathrm{and} \;\, d_n(p_n) \rightarrow \infty.
\end{align*}
In the following, we prove this claim. 
Without loss of generality, we can assume that $\Sigma_n$ is a geodesic disk centered at $p_n$. 
Set $\Sigma'_n:=\left\{ p \in \Sigma_n : d_n(p,p_n) \leq d_n(p_n)/2 \right\}$. 
Let $d'_n(p)$ be the distance from $p \in \Sigma'_n$ to the boundary of $\Sigma'_n$. 
We remark that each Gaussian curvature $K_{ds_n^2}$ is bounded on $\Sigma'_n$ and $d'_n(p) \rightarrow 0$ as $p \rightarrow \partial \Sigma'_n$. 
Then, by taking the index again if necessary, each $|K_{ds_n^2}(p)| \, d'_n(p)^2$ has a maximum at an interior point $p'_n$ of $\Sigma'_n$. 
We thereby obtain 
\begin{align*}
|K_{ds_n^2}(p'_n)| \, d'_n(p'_n)^2 
\geq |K_{ds_n^2}(p_n)| \, d'_n(p_n)^2 
= \dfrac{1}{4} \, |K_{ds_n^2}(p_n)| \, d_n(p_n)^2 
\rightarrow \infty. 
\end{align*}
So we can replace the $\Sigma_n$ by the $\Sigma'_n$ with $|K_{ds_n^2}(p'_n)| \, d'_n(p'_n)^2 \rightarrow \infty$. 
We rescale $\Sigma'_n$ to make $K_{ds_n^2}(p'_n) = -1/4$. 
By the invariance under scaling of the quantity $K(p) \, d(p)^2$, we obtain $d'_n(p'_n) = 2 \, |K_{ds_n^2}(p'_n)|^{\frac{1}{2}} \, d'_n(p'_n) \rightarrow \infty$. 
Here, without causing confusion, we use the same notation $d'_n$ to denote the geodesic distance with respect to the rescaled metric. 
We can assume again that $\Sigma'_n$ is a geodesic disk centered at $p'_n$ and let $\Sigma''_n := \left\{ p \in \Sigma'_n : d'_n(p,p'_n) \leq d'_n(p'_n)/2 \right\}$. 
Remark that $d'_n(p) \geq d'_n(p'_n)/2$ for any $p \in \Sigma''_n$. 
Then for any $p \in \Sigma''_n$, 
\begin{align*}
|K_{ds_n^2}(p)| \, \dfrac{d'_n(p'_n)^2}{4} 
\leq |K_{ds_n^2}(p)| \, d'_n(p)^2 
\leq |K_{ds_n^2}(p'_n)| \, d'_n(p'_n)^2 
= \dfrac{d'_n(p'_n)^2}{4},
\end{align*}
that is, $|K_{ds_n^2}| \leq 1$ on $\Sigma''_n$. 
Furthermore, let $d''_n(p)$ be the distance from $p \in \Sigma''_n$ to the boundary of $\Sigma''_n$. 
Then we obtain $d''_n(p'_n)=d'_n(p'_n)/2 \rightarrow \infty$ and this implies the claim. 

By taking the universal covering of each $\Sigma_n$ if necessary, we may assume that each $\Sigma_n$ is simply connected. 
From the uniformization theorem, each $\Sigma_n$ is conformally equivalent to either the unit disk $\D$ or the complex plane $\C$ and we may assume that $p_n$ maps onto $0$ for each $n \in \Z_{>0}$. 
Here we find that each $\Sigma_n$ is conformally equivalent to $\D$ by contradiction. 
Suppose that there exists $\Sigma_n$ that is conformally equivalent to $\C$. 
By Proposition \ref{lem:cpt and const}, $g_n$ is constant, that is, $K_{ds^2_n} \equiv 0$. 
This contradicts the condition $|K_{ds^2_n}(0)|=1/4$. 
Hence we consider the sequence $\{ (\D, \, f_n \, dz, \, g_n) \}_{n=1}^\infty$ of Weierstrass $m$-triples that satisfies $g_n \in \P(\D) \; (^\forall n \in \Z_{>0})$ and 
\begin{align} \label{eq:Main thm.1'}
K_{ds_n^2}(0)=-\dfrac{1}{4},
\;\, -1 \leq K_{ds_n^2} \leq 0 \; \mathrm{on} \; \D \;\; (^\forall n \in \Z_{>0})
\;\, \mathrm{and} \;\, d_n(0) \rightarrow \infty.
\end{align} 
Since $P$ is a compact property, by taking a subsequence if necessary, there exists $g \in \P(\D)$ such that $g_n \stackrel{\mathrm{sph.\;loc.}}{\rightrightarrows} g \; \mathrm{on} \;\D$. 
Then the following lemma holds:

\begin{lemma} \label{sublemma}
Let $\{ (\D, \, f_n \, dz, \, g_n) \}_{n=1}^\infty$ be a sequence of Weierstrass $m$-triples that satisfies $g_n \stackrel{\mathrm{sph.\;loc.}}{\rightrightarrows} g \; \mathrm{on} \;\D$ and \eqref{eq:Main thm.1'}, that is, 
\begin{align*}
K_{ds_n^2}(0)=-\dfrac{1}{4},
\;\, -1 \leq K_{ds_n^2} \leq 0 \; \mathrm{on} \; \D \;\; (^\forall n \in \Z_{>0}) 
\;\, \mathrm{and} \;\, d_n(0) \rightarrow \infty,
\end{align*}
where $d_n(0)$ is the geodesic distance from the origin to the boundary $\partial \D$ in the metric $ds_n^2 = (1+|g_n|^2)^m \, |f_n|^2 \, |dz|^2$. 
Then $g$ is nonconstant on $\D$.
\end{lemma}

\begin{proof}[\textit{Proof of Lemma \ref{sublemma}}]
We prove this lemma by contradiction. 
Suppose that $g$ is constant on $\D$.  
By composing with a suitable M\"{o}bius transformation if $g \equiv \infty$ on $\D$, there exists $\alpha \in \C$ such that $g \equiv \alpha$ on $\D$. 
We set $L:=|\alpha|+1>0$. 
By Lemma \ref{lem:point curvature est.}, there exists a positive constant $C=C(L,m)>0$ such that for any Weierstrass $m$-triple $(D(0;\rho), \, F \, dz, \, G)$ satisfying $|G| < L$ on $D(0;\rho)$, we have
\begin{align} \label{eq:Main thm.2} 
|K_{ds^2}(0)|^{\frac{1}{2}} \leq \dfrac{C}{d(\rho)},
\end{align}
where $d(\rho)$ is the geodesic distance from the origin to the boundary $\partial D(0;\rho)$ in the metric $ds^2 = (1+|G|^2)^m \, |F|^2 \, |dz|^2$. 
For $C$ in \eqref{eq:Main thm.2}, we take $0<r<1$ with $2 \, C < \log \, \{(1+r)/(1-r)\}$ and set $R := \log \, \{(1+r)/(1-r)\}$. 
Here, remark that $R$ is the hyperbolic distance of $D(0;r)$. \par    
By Fact \ref{fact:conv}, $g_n \stackrel{\mathrm{loc.}}{\rightrightarrows} g \; \mathrm{on} \;\D$. 
In particular, $g_n \rightrightarrows g \; \mathrm{on} \;\overline{D(0;r)}$, that is, there exists $N \in \Z_{>0}$ such that for any $n \geq N$, $|g_n| \leq |g|+1=L$ on $\overline{D(0;r)}$. 
Applying  \eqref{eq:Main thm.2} to $(D(0;r), \, f_n \, dz, \, g_n) \; (n \geq N)$, for any $n \geq N$, 
\begin{align*}
|K_{ds_n^2}(0)|^{\frac{1}{2}} \, d_n(r) \leq C,
\end{align*}
where $d_n(r)$ is the geodesic distance from the origin to the boundary $\partial D(0;r)$ in the metric $ds^2_n = (1+|g_n|^2)^m \, |f_n|^2 \, |dz|^2$. 
By \eqref{eq:Main thm.1'}, we obtain 
\begin{align} \label{eq:Main thm.3}
d_n(r) \leq 2 \, C \quad(n \geq N).
\end{align}
Set $R_n := d_n(0)$, where $d_n(0)$ is the geodesic distance from the origin to the boundary $\partial \D$ in the metric $ds^2_n$. 
Moreover, let $\partial D(0;r_n)$ be the circle in $\D$ of hyperbolic radius $R_n$, that is, 
\begin{align*}
r_n := \dfrac{e^{R_n}-1}{e^{R_n}+1} \, (<1),
\quad R_n=\log\,\dfrac{1+r_n}{1-r_n}.
\end{align*}
We parametrize by $\phi_n:D(0;r_n) \longrightarrow \D\;;\;\phi_n(w):=w/r_n$. 
In the following, let $\phi^*_n(ds^2_n)$ denote the pullback of the metric $ds^2_n$ by $\phi_n$. 
Applying Fact \ref{fac:Oss-Ru's lemma} to $(D(0;r_n),\phi_n^*(ds_n^2))$, for any $w \in \partial D(0;r_nr) (\subset D(0;r_n))$, we obtain
\begin{align*}
& \mathrel{\phantom{=}} (\mathrm{the \; distance \; between \;}0 \mathrm{\; and \;} w \mathrm{\; in \; the \; metric \;}\phi_n^*(ds_n^2)) \\ 
& \geq ( \mathrm{the \; hyperbolic \; distance \; between \;}0 \mathrm{\; and \;}w ),
\end{align*}
that is,
\begin{align} \label{eq:Main thm.4}
(\mathrm{the \; geodesic \; distance \; from \;} 0 \mathrm{\; to \;}\partial D(0;r_nr) \mathrm{\; in \; the \; metric\;}\phi_n^*(ds_n^2)) 
\geq \log\,\dfrac{1 + r_n \, r}{1 - r_n \, r}.
\end{align}
Since $r_n \rightarrow 1$ (as $n \rightarrow \infty$),  $\log \, \{(1 + r_n \, r)/(1 - r_n \, r)\} \rightarrow \log \, \{(1 + r)/(1 - r)\} = R > 2 \, C$ (as $n \rightarrow \infty$). 
By \eqref{eq:Main thm.4}, there exists $n_0 \in \Z_{\geq N}$ such that for any $n \geq n_0$, 
\begin{align*}
(\mathrm{the \; geodesic \; distance \; from \;} 0 \mathrm{\; to \;}\partial D(0;r_nr) \mathrm{\; in \; the \; metric\;}\phi_n^*(ds_n^2))
> 2 \, C.
\end{align*}
Hence, for any $n \geq n_0$, 
\begin{align*}
d_n(r)
& = (\mathrm{the \; geodesic \; distance \; from \;} 0 \mathrm{\; to \;}\partial D(0;r) \mathrm{\; in \; the \; metric\;}ds_n^2) \\
& = (\mathrm{the \; geodesic \; distance \; from \;} 0 \mathrm{\; to \;}\partial D(0;r_nr) \mathrm{\; in \; the \; metric\;}\phi_n^*(ds_n^2)) \\
& > 2 \, C.
\end{align*}
This contradicts \eqref{eq:Main thm.3} and it implies $g$ is nonconstant on $\D$. 
\end{proof} 
 
We therefore find that the limit function $g$ is nonconstant on $\D$ and it implies that the hypotheses of Lemma \ref{lem:conv of m-Riemann surf} are satisfied. 
Furthermore, $\rm(\hspace{.18em}i\hspace{.18em})$ in Lemma \ref{lem:conv of m-Riemann surf} cannot happen because $|K_{ds_n^2}(0)|=1/4$ for each $n \in \Z_{>0}$. 
Thus, by taking a subsequence if necessary, there exists a holomorphic $1$-form $f \, dz$ on $\D$ such that $ds^2 = (1+|g|^2)^m \, |f|^2 \, |dz|^2$ is a conformal metric on $\D$ and $ds_n^2 \stackrel{\mathrm{loc.}}{\rightrightarrows} ds^2 \; \mathrm{on} \;\D$. 
From \eqref{eq:Main thm.1'}, 
\begin{align*}
|K_{ds^2}(0)|=\dfrac{1}{4}, \quad -1 \leq K_{ds^2} \leq 0 \; \mathrm{on} \; \D.
\end{align*}
Moreover, we find that the conformal metric $ds^2$ is complete on $\D$. 
It is sufficient to show that for any divergent path $C$ starting from the origin in $\D$, the length of $C$ in the metric $ds^2$ is infinity. 
Since this divergent path $C$ tends to the boundary in $\D$,
\begin{align*}
\int_{C} ds_n \geq 
\inf \left\{ \int_{\gamma} ds_n\;:\;\gamma\;\mathrm{is \; a \; curve \; from}\;z=0\;\mathrm{to}\;|z|=1 \right\}=d_n(0).
\end{align*}
Recall that $ds_n^2 \stackrel{\mathrm{loc.}}{\rightrightarrows} ds^2 \; \mathrm{on} \;\D$. 
As $n \rightarrow \infty$, we obtain $\int_C ds=+\infty$. 
\end{proof}

\section{Applications}\label{sec:4}

In this section, we give several applications of our main results to surface theory. 

\subsection{Gauss map of minimal surfaces in $\R^3$} \label{sbseq:4.1}
We first review some basic facts of minimal surfaces in the Euclidean $3$-space $\R^3$. 
For more details, for example, see \cite{AFL2021, Fu1993, Ko2021, Lo1980, Oss1986, Ru2023}. 
Let $\psi = (x^1, \, x^2, \, x^3) : \Sigma \longrightarrow \R^3$ be an oriented minimal surface. 
By associating a local complex coordinate $z = u + \mathrm{i} \, v$ ($\mathrm{i} := \sqrt{-1}$) 
with each positive isothermal coordinate $(u, v)$, $\Sigma$ is considered as 
a Riemann surface whose conformal 
metric is the induced metric $ds^2$ from $\R^3$. Then
\begin{align*}
\Delta_{ds^2} \, \psi = 0
\end{align*}     
holds, that is, each coordinate function $x^{i} \; (i = 1, \, 2, \, 3)$ of minimal surface $\psi$ is harmonic. 
By applying the maximum principle, there exists no compact minimal surface without boundary,
that is, $\Sigma$ is an open Riemann surface. 
For minimal surfaces, the following claim called the Enneper--Weierstrass representation formula holds. 

\begin{fact} \cite[Lemma 8.2]{Oss1986} \label{fac:representation}
Let $\Sigma$ be an open Riemann surface and $(f \, dz, \, g)$ a pair of a holomorphic $1$-form and a meromorphic function on $\Sigma$ such that 
  \begin{itemize}
    \item $(1 + |g|^2)^2 \, |f|^2 \, |dz|^2$ is a Riemannian metric on $\Sigma$ : regularity condition,
    \item for any cycle $c \in H_{1} (\Sigma, \Z)$, $\Re \, \int_c \, (1 - g^2, \, \mathrm{i} \, (1 + g^2), \, 2 \, g) \, f \, dz = 0$ : period condition.
  \end{itemize}
Then 
  \begin{align*}
    \psi(z) := \Re \, \int_{z_0}^z (1 - g^2, \, \mathrm{i} \, (1 + g^2), \, 2 \, g) \, f \, dz
  \end{align*}
is well-defined on $\Sigma$ and gives a minimal surface in $\R^3$, where $z_0 \in \Sigma$ is a base point. Conversely, all minimal surfaces are represented in this manner. Furthermore, the induced metric is given by
  \begin{align} \label{eq-metric}
    ds^2 = (1 + |g|^2)^2 \, |f|^2 \, |dz|^2.
  \end{align}
\end{fact}

Remark that for a minimal surface $\psi : \Sigma \longrightarrow \R^3$ given by $(f \, dz, \, g)$, $g$ coincides with the composition of its Gauss map and the stereographic projection $(\pi_N)^{-1} : S^2 \longrightarrow \RC$. 
We also note that if $\Sigma$ is simply connected, the period condition is automatically satisfied.

We consider Theorem \ref{thm:m-curvature est} in the case where $m = 2$. 
By Fact \ref{fac:representation} and the fact that $\D$ is simply connected, we can obtain the following criterion to determine which compact properties for the Gauss maps of minimal surfaces in $\R^3$ satisfy $2$-curvature estimate. 
This corresponds to \cite[Theorem 3]{RA}. 

\begin{theorem} \label{min-thm-1}
Let $P$ be a compact property for $\RC$-valued holomorphic maps. 
Then $\rm(\hspace{.18em}i\hspace{.18em})$ or $\rm(\hspace{.08em}ii\hspace{.08em})$ holds:
\begin{enumerate}
\item[(\hspace{.18em}i\hspace{.18em})] $P$ satisfies $2$-curvature estimate. 
\item[(\hspace{.08em}ii\hspace{.08em})] There exists a complete minimal surface $\psi \colon \D \longrightarrow \R^3$ whose Gauss map lies in $\P(\D)$ and the Gaussian curvature with respect to the metric $ds^2$ defined by \eqref{eq-metric} satisfies $|K_{ds^2}(0)|=1/4$ and $|K_{ds^2}|\leq 1$ on $\D$. 
\end{enumerate}
\end{theorem}

Moreover, as applications of our argument in Section \ref{sbsec:3.1}, 
we obtain value distribution theoretic properties for the Gauss maps of minimal surfaces in $\R^3$. 
We first show the following Liouville-type theorem (\cite[Theorem 8.1]{Oss1986}) for the class of minimal surfaces. 
In particular, we also obtain the Bernstein theorem as a corollary of this result.

\begin{theorem} 
For any complete minimal surface $\psi : \Sigma \longrightarrow \R^3$, if the image of its Gauss map is not dense in $S^2$, then $\psi(\Sigma)$ must be a plane. 
In particular, any entire minimal graph in $\R^3$ must be a plane.   
\end{theorem}

\begin{proof}
Suppose that the image of the Gauss map is not dense in $S^2$. 
Then there exists an open set in $S^2$ which is not intersected by the image of its Gauss map.
By a suitable rotation in $\R^3$ if necessary, we may assume that $N = (0,0,1) \in S^2$ is in this open set.
Hence there exists $L>0$ such that $|g|< L$ on $\Sigma$. 
We also remark that the metric given by \eqref{eq-metric} is complete.
By applying Corollary \ref{cor:Bernstein}, $g$ is constant on $\Sigma$, that is, the surface $\psi (\Sigma)$ is a plane in $\R^3$. 
The latter claim can be proved the facts that an entire minimal graph is complete and the image of its Gauss map is contained in a hemisphere in $S^2$.
\end{proof}

Furthermore, as an application of Corollary \ref{cor:Fujimoto}, 
we can obtain the Fujimoto theorem (\cite[Corollary 1.3]{FH1988}) for the Gauss map of complete minimal surfaces in $\R^3$.  

\begin{theorem}
The Gauss map of a nonflat complete minimal surface in $\R^3$ can omit at most $4 \, (=2+2)$ values. Furthermore, this result is optimal. 
\end{theorem}

We remark that the second author, Kobayashi and Miyaoka \cite{KKM2008} gave a similar result for the Gauss map of a special class of complete minimal surfaces (this class is called the pseudo-algebraic minimal surfaces). 
We also note that Fujimoto \cite{Fu1992} obtained the best possible upper bound for the total weight of the totally ramified values of the Gauss map of complete minimal surfaces in $\R^3$. 
For more details and related topics on the total weight of the totally ramified values, see \cite{KW2023}. 

\subsection{Lorentzian Gauss map of maxfaces in $\L^3$} \label{sbseq:4.2} 
Maximal surfaces in the Lorentz--Minkowski 3-space ${\L}^{3}$ are closely related to  minimal surfaces in ${\R}^{3}$. 
We treat maximal surfaces with some admissible singularities, called {\it maxfaces}, as introduced by Umehara and Yamada \cite{UY2006}. 
We remark that maxfaces, non-branched generalized maximal surfaces in the sense of \cite{ER1992} and non-branched generalized maximal maps in the sense of \cite{IK2008} are all the same class of maximal surfaces. 
The Lorentz--Minkowski $3$-space ${\L}^{3}$ is the affine $3$-space ${\R}^{3}$ with the inner product 
\[
\langle \, , \, \rangle = -(dx^{1})^{2}+(dx^{2})^{2}+(dx^{3})^{2},
\]
where $(x^{1}, \, x^{2}, \, x^{3})$ is the canonical coordinate system of ${\R}^{3}$. 
We consider a fibration 
\[
p_{L}\colon {\C}^{3} \ni ({\zeta}^{1}, \, {\zeta}^{2}, \, {\zeta}^{3}) \longmapsto \text{Re} (- \mathrm{i} \, {\zeta}^{1}, \, {\zeta}^{2}, \, {\zeta}^{3}) \in {\L}^{3}. 
\]
The projection of null holomorphic immersions into  ${\L}^{3}$ by $p_{L}$ gives maxfaces. 
Here, a holomorphic map $F = (F_{1}, \, F_{2}, \, F_{3})\colon \Sigma \to {\C}^{3}$ is said to be {\it null} if $\{(F_{1})'\}^{2}+\{(F_{2})'\}^{2}+\{(F_{3})'\}^{2}$ vanishes identically, where $'=d / dz$ denotes the derivative with respect to a local complex coordinate $z$ of $\Sigma$. 
For maxfaces, the analogue of the Enneper--Weierstrass representation formula is known (see also \cite{Ko1983}). 

\begin{fact} \cite[Theorem 2.6]{UY2006} \label{max-EW}
Let $\Sigma$ be a Riemann surface and $(f \, dz, \, g)$ a pair of a holomorphic $1$-form and a meromorphic function on $\Sigma$ such that 
\begin{align}\label{max-nullmet}
d{\sigma}^{2} := (1+|g|^{2})^{2} \, |f|^{2} \, |dz|^{2}
\end{align}
gives a Riemannian metric on $\Sigma$ and $|g| \not \equiv 1$ on $\Sigma$. Assume that 
\begin{align}\label{maximal-period}
\mathrm{Re} \int_{c} (- 2 \, g, \, 1+g^{2}, \mathrm{i} \, (1-g^{2})) \, f \, dz = 0
\end{align}
for every cycle $c \in H_{1} (\Sigma, \Z)$. 
Then 
\begin{align*}
\psi(z) := \mathrm{Re} \int^{z}_{z_{0}} (- 2 \, g, 1+g^{2}, \mathrm{i} \, (1-g^{2})) \, f \, dz 
\end{align*}
is well-defined on $\Sigma$ and gives a maxface in $\L^3$, where $z_{0}\in {\Sigma}$ is a base point. 
Moreover, all maxfaces are obtained in this manner. 
The induced metric $ds^{2} := \psi^{\ast} \langle \, , \, \rangle$ is given by
\[
ds^{2} = (1-|g|^{2})^{2} \, |f|^{2} \, |dz|^2, 
\]
and a point $p \in \Sigma$ is a singular point of $\psi$ if and only if $|g(p)|=1$. 
\end{fact}

We remark that if $\Sigma$ is simply connected, the condition \eqref{maximal-period} is automatically satisfied. 
We call $g$ the {\it Lorentzian Gauss map} of $\psi$. 
If $\psi$ has no singularities, then $g$ coincides with the composition of the Gauss map (i.e., (Lorentzian) unit normal vector) $n \colon \Sigma \to {\H}^{2}_{\pm}$, where 
\begin{eqnarray}
{\H}^{2}_{+} &:=& \{n=(n^{1}, \, n^{2}, \, n^{3})\in \L^3\, ; \, \langle n, n \rangle = -1, \, n^{1}> 0\}, \nonumber \\ 
{\H}^{2}_{-} &:=& \{n=(n^{1}, \, n^{2}, \, n^{3})\in \L^3\, ; \, \langle n, n \rangle = -1, \, n^{1}< 0\},   \nonumber
\end{eqnarray}
and ${\H}^{2}_{\pm}:={\H}^{2}_{+}\cup {\H}^{2}_{-}$ in $\L^3$ and the stereographic projection from $(1, 0, 0)$ of the hyperboloid onto $\RC$ (see \cite[Section 1]{UY2006} for the details). 
A maxface is said to be {\it weakly complete} if the metric $d{\sigma}^{2}$ as in (\ref{max-nullmet}) is complete. 
We remark that $(1/2) \, d{\sigma}^{2}$ coincides with the pull-back of the standard metric on ${\C}^{3}$ by the null holomorphic immersion of $\psi$ (see \cite[Section 2]{UY2006}). 

By applying Theorem \ref{thm:m-curvature est} in the case where $m = 2$, we can obtain the following criterion to determine which compact properties for the Lorentzian Gauss maps of maxfaces in $\L^3$ satisfy $2$-curvature estimate.  

\begin{theorem}\label{min-thm-2}
Let $P$ be a compact property for $\RC$-valued holomorphic maps. 
Then $\rm(\hspace{.18em}i\hspace{.18em})$ or $\rm(\hspace{.08em}ii\hspace{.08em})$ holds:
\begin{enumerate}
\item[(\hspace{.18em}i\hspace{.18em})] $P$ satisfies $2$-curvature estimate. 
\item[(\hspace{.08em}ii\hspace{.08em})] There exists a weakly complete maxface $\psi \colon \D \longrightarrow \L^3$ whose Lorentzian Gauss map lies in $\P(\D)$ and the Gaussian curvature with respect to the metric $d\sigma^2$ defined by \eqref{max-nullmet} satisfies $|K_{d\sigma^2}(0)| = 1/4$ and $|K_{d\sigma^2}| \leq 1$ on $\D$. 
\end{enumerate}
\end{theorem}

Moreover, as applications of our argument in Section \ref{sbsec:3.1}, 
we obtain value distribution theoretic properties for the Lorentzian Gauss maps of maxfaces in $\L^3$. 
For instance, we can show the Calabi--Bernstein theorem (\cite{Ca1970,CY1976}) for maximal surfaces in $\L^3$.     

\begin{theorem}
Any complete maximal space-like surface in $\L^3$ must be a plane.   
\end{theorem}

\begin{proof}
Since a maximal space-like surface has no singularities, the complement of the image of its Lorentzian Gauss map $g$ contains at least the set $\{|g|=1\} \subset \RC$. 
Then, by considering $1/g$ if necessary, we may assume that $|g|<1$ on $\Sigma$. 
Furthermore, since 
\[
ds^{2} 
= (1-|g|^{2})^{2} \, |f|^{2} \, |dz|^2
\leq (1+|g|^{2})^{2} \, |f|^2 \, |dz|^2 
= d{\sigma}^{2} 
\]
and $ds^2$ is complete on $\Sigma$, $d{\sigma}^{2}$ is also complete on $\Sigma$. 
By Corollary \ref{cor:Bernstein}, $g$ is constant on $\Sigma$, that is, the surface $\psi(\Sigma)$ is a plane.  
\end{proof}

Furthermore, as an application of Corollary \ref{cor:Fujimoto}, 
we can obtain the Picard-type theorem (\cite[Theorem 4.3]{YK2013}) for this class. 

\begin{theorem}
If the Lorentzian Gauss map of a weakly complete maxface in $\L^3$ is nonconstant,  it can omit at most $4 \, (=2+2)$ values. 
Furthermore, this result is optimal. 
\end{theorem} 

\subsection{Lagrangian Gauss map of improper affine fronts in $\R^3$} \label{sbseq:4.3} 
Improper affine spheres in the affine $3$-space ${\R}^{3}$ also have similar properties to minimal surfaces in the Euclidean $3$-space. 
Mart\'inez \cite{Ma2005} discovered the correspondence between improper affine spheres and smooth special Lagrangian immersions in 
the complex $2$-space ${\C}^{2}$ and introduced the notion of {\it improper affine fronts}, that is, a class of (locally strongly convex) improper affine spheres with some admissible singularities in ${\R}^{3}$. 
We remark that this class is called ``improper affine maps'' in \cite{Ma2005}, but we call this class ``improper affine fronts'' because all of improper affine maps are wave fronts in ${\R}^{3}$ (\cite{Na2009,UY2011}). 
The differential geometry of wave fronts is discussed in \cite{SUY2022}. 
Moreover, Mart\'inez gave the following holomorphic representation for this class. 

\begin{fact} \cite[Theorem 3]{Ma2005} \label{fact-IArep}
Let $\Sigma$ be a Riemann surface and $(F, G)$ a pair of holomorphic functions on $\Sigma$ such that $|dF|^{2} + |dG|^{2}$ is positive definite on $\Sigma$ and 
\[
\mathrm{Re} \int_{c} F \, dG = 0
\] 
for every cycle $c \in H_{1} (\Sigma, \Z)$. 
Then the map $\psi \colon \Sigma \longrightarrow {\R}^{3} = \C \times {\R}$ given by 
\[
\psi := \biggl{(}G+\overline{F}, \, \dfrac{|G|^{2}-|F|^{2}}{2} + \mathrm{Re}\biggl{(} G \, F - 2 \int F \, dG \biggr{)} \biggr{)}
\]
is an improper affine front. 
Conversely, any improper affine front is given in this way. 
Moreover, we set $x :=  G + \overline{F}$ and $n := \overline{F} - G$. 
Then $L_{\psi} := x + \mathrm{i} \, n \colon \Sigma \longrightarrow {\C}^{2}$ is a special Lagrangian immersion whose induced metric $d{\tau}^{2}$ from ${\C}^{2}$ is given by 
\[
d{\tau}^{2} = 2 \, (|dF|^{2}+|dG|^{2}). 
\]
Furthermore, the affine metric $h$ of $\psi$ is expressed as $h := |dG|^{2} - |dF|^{2}$ and a singular point of $\psi$ corresponds to a 
point where $|dF| = |dG|$. 
\end{fact}

We remark that Nakajo \cite{Na2009} constructed a representation formula for indefinite improper affine spheres with some admissible singularities. 

The nontrivial part of the Gauss map of $L_{\psi} \colon \Sigma \longrightarrow {\C}^{2} \simeq {\R}^{4}$ (see \cite{CM1987}) is the meromorphic function $\nu \colon \Sigma \longrightarrow \RC$ given by 
\[
\nu := \dfrac{dF}{dG}, 
\]
which is called the {\it Lagrangian Gauss map} of $\psi$. 
An improper affine front is said to be {\it weakly complete} if the induced metric $d{\tau}^{2}$ is complete. 
Remark that 
\begin{align} \label{eq:induced-metric}
d{\tau}^{2} = 2 \, (|dF|^{2}+|dG|^{2}) = 2 \, (1+|\nu|^{2}) \, |dG|^{2}. 
\end{align}

By applying Theorem \ref{thm:m-curvature est} in the case where $m = 1$, we can obtain the following criterion to determine which compact properties for the Lagrangian Gauss maps of improper affine fronts in $\R^3$ satisfy $1$-curvature estimate. 

\begin{theorem}\label{min-thm-3}
Let $P$ be a compact property for $\RC$-valued holomorphic maps. 
Then $\rm(\hspace{.18em}i\hspace{.18em})$ or $\rm(\hspace{.08em}ii\hspace{.08em})$ holds:
\begin{enumerate}
\item[(\hspace{.18em}i\hspace{.18em})] $P$ satisfies $1$-curvature estimate. 
\item[(\hspace{.08em}ii\hspace{.08em})] There exists a weakly complete improper affine front $\psi \colon \D \longrightarrow \R^3$ whose Lagrangian Gauss map lies in $\P(\D)$ and the Gaussian curvature with respect to the metric $d\tau^2$ defined by \eqref{eq:induced-metric} satisfies $|K_{d\tau^2}(0)|=1/8$ and $|K_{d\tau^2}|\leq 1/2$ on $\D$. 
\end{enumerate}
\end{theorem}

Moreover, as applications of our argument in Section \ref{sbsec:3.1}, 
we obtain value distribution theoretic properties for the Lagrangian Gauss maps of improper affine fronts in $\R^3$. 
For instance, we can show the parametric affine Bernstein theorem (\cite{Ca1958, Jo1954}) for improper affine spheres in $\R^3$. 

\begin{theorem}
Any affine complete improper affine sphere in $\R^3$ must be an elliptic paraboloid. 
\end{theorem}

\begin{proof}
Since an improper affine sphere has no singularities, the complement of the image of its Lagrangian Gauss map $\nu$ contains at least the circle $\{ |\nu| =1 \}\subset \RC$. 
Then, by exchanging roles of $dF$ and $dG$ if necessarily, we may assume that $|\nu| < 1$, that is, $|dF| < |dG|$. 
In addition, since
\[
h = |dG|^{2}-|dF|^{2} < 2 \, (|dF|^{2}+|dG|^{2}) = d\tau^2  
\]
and $h$ is complete on $\Sigma$, $d\tau^2$ is also complete on $\Sigma$. 
By applying Corollary \ref{cor:Bernstein}, $\nu$ is constant. 
From \cite[Proposition 3.1]{KN}, $\psi (\Sigma)$ is an elliptic paraboloid.   
\end{proof}

Furthermore, as an application of Corollary \ref{cor:Fujimoto}, we can obtain the Picard-type theorem (\cite[Theorem 3.2]{KN}) for this class. 

\begin{theorem}
If the Lagrangian Gauss map of a weakly complete improper affine front in $\R^3$ is nonconstant, it can omit at most $3 \, (=1+2)$ values. 
Furthermore, this result is optimal. 
\end{theorem}

\subsection{Ratio of canonical forms of flat fronts in $\H^3$} \label{sbseq:4.4} 
For a simply connected Riemann surface $\Sigma$ and a holomorphic Legendrian immersion $\Lc\colon \Sigma \to SL(2, \C)$, the projection 
\[
\psi := \Lc \, {\Lc}^{\ast}\colon \Sigma \to {\H}^{3}
\]
gives a {\it flat front} in ${\H}^{3}$. 
Here, flat fronts in ${\H}^{3}$ are flat surfaces in ${\H}^{3}$ with some admissible singularities (see \cite{KRUY2007,KUY2004} for the definition of flat fronts in ${\H}^{3}$). 
We call $\Lc$ a {\it holomorphic lift} of $\psi$. 
Since $\Lc$ is a holomorphic Legendrian map, ${\Lc}^{-1} \, {d \, \Lc}$ is off-diagonal (see \cite{GMM2000,KUY2003,KUY2004}). 
For holomorphic $1$-forms $\omega$ and $\theta$ on $\Sigma$, if we set  
\[
{\Lc}^{-1} \, {d \, \Lc} = \left(
\begin{array}{cc}
0      & \theta  \\
\omega & 0
\end{array}
\right), 
\]
then the pull-back of the canonical Hermitian metric of $SL(2, \C)$ by $\Lc$ is represented as 
\[
ds^{2}_{\Lc}:=|\omega|^{2}+|\theta|^{2}.
\] 
We remark that $2 \, ds^{2}_{\Lc}$ coincides with the pull-back of the Sasakian metric on the unit cotangent bundle $T_{1}^{\ast} \H^3$ by the holomorphic lift $\Lc$ of $\psi$ (See \cite[Section 2]{KUY2004} for details). 
A flat front $\psi$ is said to be {\it weakly complete} if the metric $ds^{2}_{\Lc}$ is complete (\cite{KRUY2009, UY2011}). 
We define a meromorphic function on $\Sigma$ by the ratio of canonical forms 
\[
\rho := \dfrac{\theta}{\omega}. 
\]
Then a point $p \in \Sigma$ is a singular point of $\psi$ if and only if $|\rho (p)|=1$ (\cite{KRSUY2005}). 
We remark that 
\begin{align}\label{eq:flat-Sasaki}
ds^{2}_{\Lc} 
= |\omega|^{2}+|\theta|^{2} 
= (1+|\rho|^{2}) \, |\omega|^{2}. 
\end{align}

By applying Theorem \ref{thm:m-curvature est} in the case where $m = 1$, we can obtain the following criterion to determine which compact properties for the ratios of canonical forms of flat fronts in $\H^3$ satisfy $1$-curvature estimate. 
  
\begin{theorem}\label{min-thm-3}
Let $P$ be a compact property for $\RC$-valued holomorphic maps. 
Then $\rm(\hspace{.18em}i\hspace{.18em})$ or $\rm(\hspace{.08em}ii\hspace{.08em})$ holds:
\begin{enumerate}
\item[(\hspace{.18em}i\hspace{.18em})] $P$ satisfies $1$-curvature estimate. 
\item[(\hspace{.08em}ii\hspace{.08em})] There exists a weakly complete flat front $\psi \colon \D \longrightarrow \H^3$ whose ratio of canonical forms lies in $\P(\D)$ and the Gaussian curvature with respect to the metric $ds_{\Lc}^2$ defined by \eqref{eq:flat-Sasaki} satisfies $|K_{ds_{\Lc}^2}(0)|=1/4$ and $|K_{ds_{\Lc}^2}|\leq 1$ on $\D$. 
\end{enumerate}
\end{theorem}

Moreover, as applications of our argument in Section \ref{sbsec:3.1}, we obtain value distribution theoretic properties for the ratios of canonical forms of flat fronts in $\H^3$.  
For instance, we can show the following uniqueness theorem (see, for example, \cite{GMM2000, Sa, VV}) for complete flat surfaces in $\H^3$. 

\begin{theorem}
Any complete flat surface in $\H^3$ must be a horosphere or a hyperbolic cylinder. 
\end{theorem}

\begin{proof}
Since a flat surface has no singularities, the complement of the image of the ratio of canonical forms $\rho$ contains at least the circle $\{|\rho| =1\}\subset \RC$. 
Then, by exchanging roles of canonical forms if necessarily, we may assume that $|\rho|< 1$. 
Furthermore, \cite[Corollary 3.4]{KUY2004} states that a complete flat surface in $\H^3$ is also weakly complete. 
By Corollary \ref{cor:Bernstein}, we obtain that $\rho$ is constant. 
From \cite[Proposition 4.4]{KN}, it is a horosphere or a hyperbolic cylinder.    
\end{proof}

Furthermore, as an application of Corollary \ref{cor:Fujimoto}, we can obtain the Picard-type theorem (\cite[Theorem 4.5]{KN}) for this class. 

\begin{theorem}
If the ratio of canonical forms of a weakly complete flat front in $\H^3$ is nonconstant, it can omit at most $3 \, (=1+2)$ values. 
Furthermore, this result is optimal. 
\end{theorem}

\bigskip


\end{document}